\documentclass[a4paper,10pt]{article}

\usepackage{amssymb,latexsym, amsfonts, amsmath, amsthm}
\usepackage{graphicx}
\usepackage{setspace}
\usepackage{enumerate}
\usepackage{wrapfig}
\usepackage{mathdots}
\usepackage{lscape}
\usepackage[paperwidth=210mm, paperheight=297mm, top=1in, bottom=1in, left=1in, right=1in, asymmetric, marginparwidth=0.6in]{geometry}
\usepackage{rotating}
\usepackage{epstopdf}
\usepackage[backgroundcolor=green!40, linecolor=green!40]{todonotes}
\usepackage{tikz}
\usepackage{todonotes}
\usepackage{MnSymbol}
\DeclareFontFamily{U}{matha}{\hyphenchar\font45}
\DeclareFontShape{U}{matha}{m}{n}{
      <5> <6> <7> <8> <9> <10> gen * matha
      <10.95> matha10 <12> <14.4> <17.28> <20.74> <24.88> matha12
      }{}
\DeclareSymbolFont{matha}{U}{matha}{m}{n}
\DeclareFontFamily{U}{mathx}{\hyphenchar\font45}
\DeclareFontShape{U}{mathx}{m}{n}{
      <5> <6> <7> <8> <9> <10>
      <10.95> <12> <14.4> <17.28> <20.74> <24.88>
      mathx10
      }{}
\DeclareSymbolFont{mathx}{U}{mathx}{m}{n}

\DeclareMathSymbol{\obot}         {2}{matha}{"6B}
\DeclareMathSymbol{\bigobot}       {1}{mathx}{"CB}

\newcommand{\DistTo}{\xrightarrow{
   \,\smash{\raisebox{-0.65ex}{\ensuremath{\scriptstyle\sim}}}\,}}

\newcommand{\Lie}{\text{Lie}}

\usetikzlibrary{arrows,calc,matrix,fit}

\theoremstyle{definition}
\newtheorem{theorem}[equation]{Theorem}
\newtheorem{proposition}[equation]{Proposition}

\newtheorem{definition}[equation]{Definition}
\newtheorem{lemma}[equation]{Lemma}
\newtheorem{corollary}[equation]{Corollary}
\newtheorem{remark}[equation]{Remark}

\usepackage[utf8]{inputenc}

\newcommand{\diman}{\operatorname{diman}}

\numberwithin{equation}{section}

\newcommand{\be}{\begin{enumerate}}
\newcommand{\ee}{\end{enumerate}}
\newcommand{\bi}{\begin{itemize}}
\newcommand{\ei}{\end{itemize}}
\newcommand{\beq}{\begin{equation}}
\newcommand{\eeq}{\end{equation}}
\newcommand{\mf}{\mathfrak}
\newcommand{\ra}{\rightarrow }

\def\End{\operatorname{End}}
\def\Aut{\operatorname{Aut}}

\def\U{\operatorname{U}}

\def\Hom{\operatorname{Hom}}

\def\dim{\operatorname{dim}}

\def\id{\operatorname{id}}

\def\GL{\operatorname{GL}}
\def\SL{\operatorname{SL}}

\def\diag{\operatorname{diag}}

\def\ddag{{\,\dag}}
\def\Tr{\operatorname{Tr}}

\def\Res{\operatorname{Res}}

\def\Nrd{\operatorname{Nrd}}

\def\CC{\mathbb{C}}

\def\NN{\mathbb{N}}

\def\ZZ{\mathbb{Z}} 
\def\bbZ{\mathbb{Z}}

\def\bbG{\mathbb{G}}
\def\bbN{\mathbb{N}}

\def\C{{\mathcal{C}}}

\def\N{{\rm N}}

\def\p{\mathfrak{p}}

\def\>{\geqslant}
\def\<{\leqslant}

\def\({\left(}
\def\){\right)}

\def\presuper#1#2%
  {\mathop{}%
   \mathopen{\vphantom{#2}}^{#1}%
   \kern-\scriptspace%
   #2}
\makeatletter
\tikzstyle{notestyleraw}=[
  draw=\@todonotes@currentbordercolor,
  fill=\@todonotes@currentbackgroundcolor,
  line width=0.5pt,
  text width=\@todonotes@textwidth+9ex,
  inner sep=0.8ex,
  rounded corners=4pt
]
\makeatother

\newcommand{\piD}{\pi_D}
\newcommand{\piF}{\pi_F}
\def\C{{\rm C}}
\newcommand{\ti}[1]{\tilde{#1}}
\newcommand{\trd}{\text{trd}}
\newcommand{\tr}{\text{tr}}

\newcommand{\mc}{\mathcal}

\newcommand{\tiG}{\ti{G}}

\title{Semisimple characters for inner forms II: Quaternionic inner forms of classical groups}
\date{December 2017}
\author{D. Skodlerack}

\begin{document}

\maketitle

\begin{abstract}
In this article we consider a quaternionic inner form~$G$ of a~$p$-adic classical group defined over a non-archimedian local field of odd residue characteristic.
We construct all full self-dual semisimple characters for~$G$ and we classify their intertwining classes using endo-parameters. Further 
we prove an intertwining and conjugacy theorem for self-dual semisimple characters. We give the formulas for the set of intertwiners between self-dual 
semisimple characters. We count all~$G$-intertwining classes of self-dual semisimple characters which lift to the same~$\tiG$-intertwining class of a semisimple character for 
the ambient general linear group~$\tiG$ for~$G$. 

%
\end{abstract}

%
\section{Introduction}

Self-dual Semisimple characters play an important role in the classification of smooth representations of~$p$-adic classical groups in odd residue characteristic and 
for the explicit understanding of the Local Langlands correspondence and the Jaquet-Langlands correspondence. Throughout the introduction and the paper we only consider 
non-archimedian local fields of odd residue characteristic. One of them should be~$F$. To start the history of the development of semisimple characters we start with 
work of Bushnell and Kutzko~\cite{bushnellKutzko:93} who classified all irreducible complex representations of~$\GL_m(F)$ via types, the latter constructed using 
simple characters. This work which was generalized by S\'echerre and Stevens to~$\GL_m(D)$,~$D$ a non-split quaternion algebra of~$F$, see~\cite{secherreStevensVI:12}, here 
as well using simple characters. 
Self-dual semisimple characters were used for the exhaustiveness proof for the classification of all cuspidal irreducible representations for~$p$-adic classical 
groups in Stevens work~\cite{stevens:08} in introducing cuspidal types. A study of the intertwining of these characters was needed, 
see~\cite{kurinczukSkodlerackStevens:16}, to prove that Stevens' construction leads to equivalent cuspidal irreducible representations only
if the cuspidal types are conjugate up to equivalence. Motivated by that result the author generalized semisimple characters to~$\GL_m(D)$ in~\cite{skodlerack:17-1}
and to quaternionic inner forms of~$p$-adic classical groups in this paper. 

Let~$(D,\rho)$ be skew-field with an orthogonal anti-involution and~$(V,h)$ be~$\epsilon$-hermitian form with respect to~$\rho$ on a finite dimensional vector space~$V$. 
We consider the set~$G=\U(h)$ of isometries of~$h$ in the ambient general linear group~$\tiG$. A semisimple character of~$\tiG$ 
is a character on a compact open subgroup of~$\tiG$ which is constructed from a datum~$\Delta:=[\Lambda,n,r,\beta]$ such that the following holds. 
\begin{itemize}
 \item $\beta$ is an element of~$\End_D(V)$ generating a product~$E=\prod_{i\in I}E_i$ of fields over~$F$. It also gives a direct sum decomposition of~$V$ 
 into~$E_i\otimes D$-modules.  
 \item $\Lambda$ is an~$o_D$-lattice sequence in~$V$, i.e. a point of the Bruhat-Tits building~$B(\tiG)$ of~$\tiG$ with rational barycentric coordinates, which is in the image of
 the embedding
 \[\prod_i B(\tiG_i)\ra B(\tiG)\]
 where~$\tiG_i=\Aut_{E_i\otimes_F D}(V^i)$. In particular~$\Lambda$ splits into~$\oplus_{i\in I}\Lambda^i$.
 \item The integers~$n$ and~$r$,~$n>r$, are non-negative and indicate on which ``level'' the characters should be defined ($r$) and should be trivial ($n$).  
 \item Some more conditions which simplify the calculation of the intertwining.  
\end{itemize}

We call the set of semisimple characters defined by~$\Delta$ by~$\C(\Delta)$. All the characters in~$\C(\Delta)$ have the same domain which we call~$H(\Delta)$. 
To define self-dual semisimple characters we consider~$\Lambda^i$ to be a point in the building of~$(\prod_i\tiG_i)\cap G$ and~$\beta$ to be an element of the Lie-algebra of~$G$.
Now the self-dual semisimple characters for~$G$ are the restrictions of elements of~$\C(\Delta)$ to~$H(\Delta)\cap G$. We call the set of them by~$\C_-(\Delta)$. 

The first steps in the study of self-dual semisimple characters give the following results:
\begin{itemize}
 \item A nice intertwining formula, i.e. the set of elements in~$G$ which intertwine~$\theta_-\in\C_-(\Delta)$ is up to multiplication from left and right by a 
 compact subgroup of the~$1$-units of~$\Lambda$ the centralizer of~$\beta$ in~$G$, see Theorem~\ref{thmIntertwiningCharG}.
 \item Intertwining is an equivalence relation for full (this means~$r=0$) self-dual semisimple characters for~$G$, see Corollary~\ref{corTransIntertwiningCharG}.
 \item We have an intertwining and conjugacy theorem, see Theorem~\ref{thmIntConCharG}, which we explain now. 
\end{itemize}

The set of self-dual semisimple character~$\C_-(\Delta)$ comes along with an action of~$\sigma$ on the index set which leads to a disjoint union~$I=I_0\cup I_+\cup I_-$
where~$I_0$ is the set of~$\sigma$ fixed points and~$I_+$ is a section through the~$\sigma$-orbits of length~$2$. 
Given two full semisimple characters~$\theta_-\in\C_-(\Delta)$ and~$\theta'_-\in\C_-(\Delta')$ which intertwine by some element of~$G$, they possess a bijection~$\zeta: I\ra I'$ 
between the index sets such that there is an element of~$G\cap \prod_i\End_D(V^i,V^{\zeta(i)})$ which intertwines~$\theta_-$ with~$\theta'_-$. 
Further the intertwining gives a bijection~$\bar{\zeta}$ between the residue algebras of~$F[\beta]$ and~$F[\beta']$. Now Theorem~\ref{thmIntConCharG} states
\begin{theorem}[see~\ref{thmIntConCharG}]
 Suppose there is an element~$t\in G$ such that~$t\Lambda^i$ is equal to~$\Lambda^{\zeta(i)}$ and that the conjugation with~$t$ verifies~$\bar{\zeta}.$
 Then there is an element~$g\in G$ such that~$g.\theta_-=\theta'_-$. 
\end{theorem}

In the second part we parametrize the intertwining classes of self-dual semisimple characters using endo-parameters. 
The idea is to break up~$\theta_-$ in elementary self-dual pieces, i.e. in self-dual semisimple characters where the index set is just one~$\sigma$-orbit,
~$\theta_{i,-}=\theta_-|_{H(\Delta)\cap\tiG_i}$, for~$i\in I_0$, and~$\theta_{i,-}=\theta_-|_{H(\Delta)\cap\tiG_{i,\sigma(i)}}$, for~$i\in I_+$. 
To every elementary character we attach an endo-class, see after Definition~\ref{defselfdualpss}. 
We denote the set of all elementary endo-classes by~$\mc{E}_-$. Further we can attach to any~$h_i$,~$i\in I_0$, 
an~$\epsilon$-hermitian~$\sigma_{E_i}\otimes\rho$-form~$\ti{h}_{\beta_i}$ which corresponds to some element~$t_i$ of the Witt group~$W_\epsilon(\sigma_{E_i}\otimes\rho)$ which we
call Witt tower. 
We introduce an equivalence relation of the set these kind of pairs~$(\gamma,t)$, see section~\ref{sectionEndoParameter}, 
and we call the equivalence classes~$(\rho,\epsilon)$-Witt types and the set of Witt types is denoted by~$\mc{W}_{\rho,\epsilon}$. 
An endo-parameter is a map of finite support
\[f_-=(f_1,f_2):\ \mc{E}_-\ra \bbN_0\times \mc{W}_{\rho,\epsilon}\]
such that for simple~$c_-\in \mc{E}_-$ the value~$f_1(c_-)$ essentially plays the role of a Witt index, and~$f_2(c_-)$ is a Witt type which occurs in~$c_-$ (note that in~$c_-$ can occur several Witt types), 
and such that for non-simple~$c_-$ we have that~$f_2(c_-)$ is hyperbolic and~$f_1(c_-)$ is a certain degree of~$c_-$.
These endo-parameters classify intertwining classes of self-dual semisimple characters, see Theorem~\ref{thmEndoparameter}. 

At the end in the appendix we calculate the number of~$G$-intertwining classes of self-dual semisimple characters whose semisimple lifts are in the 
same~$\tiG$-intertwining class.

I thank Shaun Stevens for his interest and remarks concerning this article.

\section{Quaternionic inner Forms of~$p$-adic classical groups}
\subsection{Fixing notation}\label{subsecFirstNotations}

At first, this article is the second in a series of articles where the first one is~\cite{skodlerack:17-1}
. There will be only one difference in the notation, see Remark~\ref{remNotationDiff}.
Let~$F$ be a non-archimedean local field of odd residual characteristic. We use the usual notation~$o_F,\mf{p}_F,\kappa_F$ and~$\nu_F$ for
the valuation ring, the valuation ideal, the residue field and the normalized valuation of~$F$, the image of~$\nu_F$ being~$\bbZ$, and we use 
similar notation for other non-archimedean local skewfields. Further we fix a non-split quaternion algebra~$D$ with centre~$F$ together with an orthogonal 
anti-involution~$\rho$ on~$D$, i.e. an~$F$-linear automorphism of~$D$ which satisfies~$\rho(xy)=\rho(y)\rho(x)$. 
We can choose~$\rho$ such that there is an unramified field extension~$L|F$ and a uniformizer~$\pi_D$ in~$D$ both point-wise fixed by~$\rho$ such 
that~$\pi_D$ normalizes~$L$. We denote the non-trivial automorphism of~$L|F$ by~$\tau$. The square of~$\pi_D$ is a uniformizer of~$F$ and we
denote it by~$\pi_F$. 

We further fix a finite dimensional non-zero~right-$D$ vector space~$V$ and an~$\epsilon$-hermitian form~$h$ on~$V$,~$\epsilon\in\{-1,1\}$, i.e. a~$\bbZ$-bilinear form such that: \[h(vx,wy)=\epsilon\rho(x)\rho(h(w,v))y,\] for all~$x,y\in D$ and~$v,w\in V$. 
We denote by~$G$ the set
\[\U(h)=\{g\in\Aut_D(V)|~h(gv,gw)=h(v,w)\text{ for all }v,w\in V\}\] of isometries of~$h$. 
 We write~$\sigma_h$ for the adjoint anti-involution of~$h$,~$\tilde{G}$ for the ambient general 
linear group~$\Aut_D(V)$ and~$A$ for the ring of~$D$-linear endomorphisms of~$V$.  
\begin{remark}\label{remNotationDiff}
There will be only one difference in the notation between~\cite{skodlerack:17-1} and this article. Precisely~$\ti{G}$ denotes the group~$\GL_D(V)$ and~$G$
denotes the classical group in question. 
\end{remark}

\subsection{$L$-rational points}

The group~$G$ is the set of~$F$-rational points of an~$F$-form of a symplectic or an orthogonal algebraic group~$\bbG$. 
In this section we will describe the set~$\bbG(L)$ by an hermitian~$L$-form on~$V$. 
The set of~$L$-rational points of~$\bbG$ is given by the anti-involution~$\sigma_h\otimes_F \id$ on~$\End_D(V)\otimes_F L$,
and the latter~$L$-algebra is canonically isomorphic to~$\End_L(V)$ via 
\[\Phi:\End_D(V)\otimes_F L\ra \End_L(V),\ f\otimes_F l\mapsto (v\mapsto f(vl)).\]
The group of~$L$-rational points of~$\bbG$ is algebraically isomorphic to 
\[\{g\in \End_L(V)|\ \Phi((\sigma_h\otimes_F \id)(\Phi^{-1}(g))) g=1\}.\]
We identify this set with~$\bbG(L)$.
By \cite{broussousStevens:09} there is a unique~$\epsilon$-hermitian~$L$-form~$h_L$ on~$V$ such that
\[\tr_{L|F}\circ h_L=\trd_{D|F}\circ h.\]

\begin{proposition}
$\bbG(L)$ is equal to~$\U(h_L)$.
\end{proposition}

\begin{proof}
 We need to show that~$\sigma_{h_L}$ is equal to the push forward of~$\sigma_h\otimes_{F}\id$ via~$\Phi$. Take~$v,w\in V,f\in \End_D(V)$ and~$l\in L$. 
 Then,
 \begin{eqnarray*}
  \tr_{L|F}(h_L(\Phi(f\otimes_Fl)v,w)) &=& \tr_{L|F}(h_L(f(v)l,w))\\
  &=& \tr_{L|F}(h_L(f(v),wl))\\
  &=& \trd_{D|F}(h(f(v),wl))\\
  &=& \trd_{D|F}(h(v,\sigma_h(f)(wl)))\\
  &=& \tr_{L|F}(h_L(v,\Phi((\sigma_h\otimes_F\id)(f\otimes l))(w)))\\
 \end{eqnarray*}
\end{proof}

\subsection{Witt group of a local non-split division algebra}\label{subsecWitt}
Let~$\tilde{D}$ be a central division algebra over~$F$ of finite degree, with anti-involution~$\tilde{\rho}$, such that~$(\tilde{D},\tilde{\rho})$ 
is orthogonal or unitary. Then the Witt group of~$\tilde{D}$ with respect to~$\epsilon$ and~$\tilde{\rho}$ is the set of  equivalence classes 
of~$\epsilon$-hermitian forms 
\[\tilde{h}:\ \tilde{V}\times\tilde{V}\ra (\tilde{D},\ti{\rho})\]
on finite dimensional~$\ti{D}$-vector spaces~$\ti{V}$, where two forms are equivalent, if they have isomorphic anisotropic components in there Witt-decomposition. We write~$\tilde{h}_\equiv$ for the classes.
We write~$W_{\epsilon,\tilde{\rho}}(\tilde{D})$ for the Witt group. We call this Witt group 
\begin{itemize}
 \item {\it orthogonal} if~$\epsilon=1$ and~$\tilde{\rho}|_F=\id_F$,
 \item {\it symplectic} if~$\epsilon=-1$ and~$\tilde{\rho}|_F=\id_F$ and 
 \item {\it unitary} if~$\tilde{\rho}|_F\neq\id_F$.
\end{itemize}

If we have a Gram matrix~$M$ for~$\tilde{h}$ we also write~$\langle M\rangle$ for the form with Gram matrix~$M$. If~$f$ is a symmetric or skew-symmetric element of~$\End_{\tilde{D}}(\tilde{V})$
 then we write~$\tilde{h}^f$ for the form which maps~$(\tilde{v},\tilde{w})\in \tilde{V}^2$ to~$\tilde{h}(\tilde{v},f(\tilde{w}))$ and we call~$\ti{h}^f$ the~\emph{twist} of~$\ti{h}$ by~$f$. 

We refer to section 6 of \cite{skodlerackStevens:16} for the description of the Witt group in the case where~$\tilde{D}$ is abelian.
We write~$C_n$ for the cyclic group of order~$n$. 
In the case~$(\ti{D},\ti{\rho})=(D,\rho)$, see~\ref{subsecFirstNotations}, we have the following result for the two Witt groups.
Recall, that~$D$ is not abelian.

\begin{proposition}\label{propWittGroupD}
\begin{enumerate}
 \item The orthogonal Witt group of~$(D,\rho)$ is isomorphic to an elementary~$2$-group of order~$8$. A set of generators is given by 
 \[\langle 1\rangle_\equiv,\ \langle \pi_D\rangle_\equiv,\ \langle \alpha \rangle_\equiv,\]
where~$\alpha$ is a non-square unit of~$L$. We could take~$\alpha$ skew-symmetric with respect to~$\tau$ if~$-1$ is not a square of~$F$. 
 \item The symplectic Witt group is a cyclic group of order two. The non-hyperbolic class is~$\langle \pi_D l_s\rangle_\equiv $, where ~$l_s$ is a non-zero element of~$L$ which 
 is skew-symmetric with respect to~$\tau$. 
\end{enumerate}
\end{proposition}

At first we need the next lemmas. 
\begin{lemma}\label{lemconjFEforWD}
For every symmetric element~$d\in D^\times$ and field extension~$F'|F$ which is fixed point-wise by~$\rho$ such that~$F[d]|F$ is isomorphic to~$F'|F$  there is an element~$g$ of 
$D^\times$ such that~$gF[d]\rho(g)=F'$. In fact if~$d\not\in F$ every element which conjugates the first field extension to the second works. 
\end{lemma}

\begin{proof}
In the case~$d\in F$ we take~$g=1$. So we only need to consider the case~$d\not\in F$, i.e. where~$F[d]|F$ has degree~$2$. 
By Skolem-Noether there is an element~$g$ conjugating~$F[d]|F$ to~$F'|F$. The element~$gdg^{-1}$ is symmetric by assumption. So, 
$\rho(g)g$ centralizes~$F[d]$ and is thus an element of~$F[d]$ and we obtain that~$gF[d]g^{-1}$ is equal to~$gF[d]\rho(g)$.
\end{proof}

\begin{lemma}\label{lemFtimesd}
For every symmetric or skew-symmetric element~$d\in D^\times$ and~$x\in F^\times$ the Witt classes~$\langle d \rangle_\equiv$ and~$\langle dx \rangle_\equiv$ are equal. 
\end{lemma}

\begin{proof}
 We have to find an element~$y$ of~$D$ such that~$\rho(y)dy$ is equal to~$dx$. For the first case let us assume that 
~$d$ is~$\pi_D l_s$,~$\piD$ or in~$L^\times$. Using elements of~$y\in L$ we can obtain all~$xd$ for~$x\in F^\times$ with~$2|\nu_F(x)$. If we 
 use~$y\in L^\times\piD$ we obtain every~$xd$ for~$x\in F^\times$ with odd valuation. A general symmetric~$d$ generates a field extension
 which is generated by an element~$d'$ of~$Fl_s+F\piD$. If~$d'$ has an odd~$D$-valuation then~$F[d]$ is isomorphic to~$F[\piD]$ and if~$d'$ has an even~$D$-valuation
 then~$F[d]|F$ is unramified. Thus Lemma \ref{lemconjFEforWD} finishes the proof. 
\end{proof}

We say that two elements~$d$ and~$d'$ of~$D$ are congruent to each other modulo~$\nu_D$ if~$d-d'$ is an element of~$\p_D d$. 
This is an equivalence relation, and two non-zero elements~$d$ and~$d'$ are congruent modulo~$\nu_D$ if and only if~$dd'^{-1}$ is an 
element of~$1+\p_D$.

\begin{lemma}\label{lemCongrBasis}
Let~$a_i,\ i=1,2,3,4$ be an~$F$-splitting basis of~$D$ and suppose that two non-zero elements ~$d=\sum_i x_ia_i$ and 
$d'=\sum_i y_ia_i$ satisfy~$d d'^{-1}\in 1+\mf{p}_D$,~$x_i,y_i\in F$. Then, if~$\nu_D(x_1a_1)=\nu_D(d)$ then~$x_1y_1^{-1}\in 1+\mf{p}_F$.
\end{lemma}

\begin{proof}
We have~$\nu_D(d-d')>\nu_D(d)$ and therefore~$\nu(x_ia_i-y_ia_i)>\nu_D(d)$ for all~$i$ because~$(a_i)$ is a splitting basis for~$D$. Thus, 
$\nu_D(x_1a_1-y_1a_1)>\nu_D(x_1a_1)$ if~$\nu_D(x_1a_1)=\nu_D(d)$, i.e.~$\nu_F(x_1-y_1)>\nu_F(x_1)$.  
\end{proof}

\begin{lemma}\label{lemCongrElementsd}
Suppose~$d, d'\in D^\times$ are two symmetric or skew-symmetric elements of~$D$ which are congruent to each other modulo~$\nu_D$. Then 
\begin{enumerate}
 \item If~$d$ and~$d'$ are not elements of~$F^\times (1+\mf{p}_D)$ then~$F[d]|F$ and~$F[d']|F$ are isomorphic and there is a~$g\in D^\times$ which conjugates the first field extension \label{lemCongrElementsd.i}
 to the second such that~$gdg^{-1} d'^{-1}\in 1+\mf{p}_{F[d']}$.
 \item The Witt classes~$\langle d\rangle_\equiv$ and~$\langle d'\rangle_\equiv$ are equal. \label{lemCongrElementsd.ii}
\end{enumerate}
\end{lemma}

\begin{proof}
The residue characteristic is odd, so either both are skew-symmetric or both are symmetric. 

Case 1: Let us first consider the case where~$d$ and~$d'$ commute. It is worth to remark that this is already implied if both elements are skew-symmetric, because~$\rho$ is orthogonal. Now, then~$d$ and~$d'$ are elements of a~$\rho$-invariant field~$F'$, and we have~$d=ud'$ for some~$\rho$-symmetric element~$u$ of~$1+\mf{p}_{F'}$, using the congruence of~$d$ to~$d'$. The residue characteristic is odd and thus~$u$ is a square of a~$\rho$-symmetric element 
$v$ of~$1+\mf{p}_{F'}$ and thus~$\langle d \rangle_\equiv$ is equal to~$\langle d' \rangle_\equiv$. This proves~\ref{lemCongrElementsd.ii} in this case, and in~\ref{lemCongrElementsd.i} we can take~$g=1$.

Case 2: Suppose now that~$d$ and~$d'$ are symmetric and do not commute. We have two sub-cases. 

Case 2.1: At first let us assume that~$d$ is an element of~$F^\times (1+\mf{p}_D)$, say~$dx^{-1}\in (1+\mf{p}_D)$, for some~$x\in F^\times$. 
The elements~$d$ and~$x$ commute and therefore~$\langle d \rangle_\equiv$ is equal to~$\langle x \rangle_\equiv$, by Case 1, and similar we have~$\langle x\rangle_\equiv=\langle d'\rangle_\equiv$. So, the Witt classes of~$\langle d\rangle$ and~$\langle d'\rangle$ are the same.
It proves~\ref{lemCongrElementsd.ii}. The statement~\ref{lemCongrElementsd.i} is empty in this case. 

Case 2.2: Let us assume that~$d$ is not an element of~$F^\times (1+\mf{p}_D)$. $d'$ is congruent to~$d$ and thus it is not an element 
of~$F^\times (1+\mf{p}_D)$ either.  
\begin{enumerate}
 \item[~\ref{lemCongrElementsd.i}]  We apply Lemma~\ref{lemCongrBasis}  on 
\[d=x_1+x_2l_s+x_3\pi_D,\ d'=y_1+y_2l_s+y_3\pi_D\]
to obtain~$d-x_1$ and~$d'-y_1$ are congruent modulo~$\nu_D$. Indeed, either~$\nu_D(d)=\nu_D(x_1)$ and thus the~$x_1-y_1$ is an element of~$\p_D x_1$,  by Lemma~\ref{lemCongrBasis}, or~$\nu_D(d)<\nu_D(x_1)$ and all four elements~$d,d',d-x_1$ and~$d'-y_1$ are congruent to each other modulo~$\nu_D$. Nevertheless, the difference of~$d-x_1$ with~$d'-y_1$ must be an element of~$\p_D(d-x_1)$, because
\[\p_Dd=\p_D(d-x_1)\supseteq\p_Dx_1,\]
by~$\nu_D(d)=\nu_D(d-x_1)\leq\nu_D(x_1)$. Thus, we can assume for the proof of~\ref{lemCongrElementsd.i} that~$d$ and~$d'$ are elements of~$Fl_s+F\piD$. 
But then~$d^2$ and~$d'^2$ are congruent elements of~$F$, i.e. there is a one-unit~$v$ of~$F$ such that~$v^2d^2=d'^2$, because the residue characteristic is odd. Both,~$vd$ and~$d'$ are not elements of~$F$, and therefore~$vd$ is conjugate to~$d'$ by an element of~$D^\times$ by Skolem-Noether.
\item[~\ref{lemCongrElementsd.ii}] We know that there is an element~$g$ satisfying the first assertion. The element~$g$ is congruent to an element of~$L$ or to an element of~$L\piD$. Thus 
$\rho(g)g$, an element of~$F[d]$, is congruent to an element of the form~$\piF^i x^2$ with some unit~$x$ of~$L$. 
We claim that~$\rho(g)g$ is of the form~$ay^2$ for some~$a\in F^\times$ and~$y\in o_{F[d]}^\times$. If~$F[d]|F$ is ramified then~$\rho(g)g$ is congruent to an element of~$F$ because~$\nu_{F[d]}(\rho(g)g)$  is even and Hensel's Lemma implies the claim. 
If~$F[d]|F$ is unramified then
\[\pi_F^{-i}\rho(g)g=\rho(\piD^{-i}g)\piD^{-i}g.\] 
So, its residue class is a square and hence the claim. Thus we get
\[\langle d\rangle_\equiv=\langle gd\rho(g)\rangle_\equiv=\langle gdy^2g^{-1}\rangle_\equiv=\langle gyg^{-1}gdg^{-1}gyg^{-1}\rangle_\equiv.\] 
The latter term is~$\langle gdg^{-1} \rangle_\equiv$, because~$gyg^{-1}$ is~$\rho$-symmetric. Now,~$gdg^{-1}$ and~$d'$ are~$\rho$-symmetric, congruent and commute, and hence we get~\ref{lemCongrElementsd.ii} by Case 1. 
\end{enumerate}
\end{proof}

\begin{remark}
 The proof of Lemma \ref{lemCongrElementsd} shows that if two~$\rho$-symmetric elements~$d$ and~$d'$ are conjugate by an element of~$D^{\times}$ then
 they define the same Witt class.
\end{remark}

\begin{proof}[Proof of Proposition \ref{propWittGroupD}]
Using \cite[1.14]{bruhatTitsIV:87} every signed hermitian space~$(h,V)$ over~$D$ has a Witt basis, i.e. a basis which has a Gram matrix with two diagonal blocks:
\be
\item An anti-diagonal matrix having~$1$ and~$\epsilon$ in the anti-diagonal (block~$(1,1)$), 
\item A diagonal matrix which is the Gram matrix of an anisotropic subspace of~$V$ (block~$(2,2)$). 
\item The blocks~$(2,1)$ and~$(1,2)$ have only zero entries. 
\ee

Thus to classify all equivalence classes of signed hermitian forms, one only needs to classify the possible diagonal blocks~$(2,2)$, and for them only the diagonal entries.
The~$F$-vector space of skew-symmetric elements of~$D$ is~$F\piD l_s$. Thus, from Lemma \ref{lemFtimesd} follows that there is only one non-trivial 
Witt class in the symplectic case. The~$F$-vector space of symmetric elements of~$D$ is~$L+\piD F$. By Lemma \ref{lemCongrElementsd} and Lemma~\ref{lemFtimesd}, the class~$\langle d\rangle_\equiv$ is
$\langle \piD\rangle_\equiv$ or~$\langle 1\rangle_\equiv$ or~$\langle \alpha\rangle_\equiv$. They are pairwise different because of valuation reasons and because~$\alpha$ cannot be of the form~$\rho(x)x$ because the residue class 
of a unit of the form~$\rho(x)x$ is a square.
\end{proof}

\subsection{$G\subseteq \SL_D(V)$}
An element~$g$ of~$G$ satisfies~$g\sigma_h(g)=1$, and~$\rho|_F=\id$ leaves for~$\Nrd_{A|F}(g)$ only~$1$ or~$-1$.
It is remarkable that in fact all elements of~$G$ have reduced norm~$1$. We are going to prove this fact in this section. 

\begin{proposition}\label{propNrd1}
Every element of~$G$ has reduced norm~$1$. 
\end{proposition}

\begin{proof}
 We prove the assertion by induction on~$m=\dim_DV$. 
 We only need to consider orthogonal groups, because in the symplectic case the group~$\bbG(L)$ is a split symplectic group where every element 
 has determinant~$1$.
 Thus let us assume that~$G$ is an inner form of an orthogonal group. 
 
 Induction start $(m=1)$: Let~$g$ be an element of~$G$. We take the canonical isomorphism from~$D$ to~$\End_DD$. 
 Then~$F[g]$ is a field in~$D$, which is invariant under the action of~$\sigma_h$. If~$\sigma_h(g)=g$ then~$g^2=1$, because~$g\in G$ and thus~$g\in\{1,-1\}$, i.e.
 the reduced norm of~$g$ would be~$1$. If~$\sigma_h(g)\neq g$ then~$\sigma_h|_{F[g]}$ is the Galois generator of~$F[g]|F$. Thus~$\Nrd(g)=N_{F[g]|F}(g)=\sigma_h(g)g=1$. 
 
 Induction step~$(m>1)$: We consider~$F[g]$ generated by an element~$g$ of~$G$ . The minimal polynomial~$\mu_{g,F}$ of~$g$ over~$F$ has a prime factorization
 \[\mu_{g,F}=P^{\nu_1}_1\ldots P^{\nu_l}_l,\]
 which gives a decomposition:
 \[F[g]\cong  F[X]/P^{\nu_1}_1\ldots F[X]/P^{\nu_l}_l\]
The factors  are permuted by~$\sigma_h$. For every orbit of this action we get a~$\sigma_h$-fixed idempotent. In the case of at least two orbits we conclude that~$g$
is an element of the product of at least two quaternionic inner forms of classical groups and the induction hypothesis implies that the reduced norm of the restrictions of~$g$ to the
factors is~$1$. Thus we are left with the case of one orbit.

Case~$\mu_{g,F}=P^{\nu_1}_1P^{\nu_2}_2$ and~$\sigma_h$ flips~$F[X]/P^{\nu_1}_1$ and~$F[X]/P^{\nu_2}_2$: The primitive idempotents~$1^1$  and~$1^2$ of~$F[g]$ 
give rise to decompositions~$V^1+V^2$ of~$V$ and~$g_1+g_2$ of~$g$. From~$g\in G$ 
follows~$g_1^{-1}=\sigma_h(g_2)$. Thus
\[1^1P_2(g_1^{-1})^{\nu_2}=1^1P_2(\sigma_h(g_2))^{\nu_2}=\sigma_h(1^2P_2(g_2)^{\nu_2})=\sigma_h(1^20)=0.\]
We put~$b_0=P_2(0)$. Note that~$b_0$ is non-zero. Thus,  the polynomials $\frac{1}{b_0}P_2(X^{-1})X^{\deg P_2}$ and~$P_1$ coincide by symmetry and minimality. 
Further, the reduced characteristic polynomials of~$g$ and~$g^{-1}$ coincide by~$\sigma_h(g)=g^{-1}$, we call the polynomial~$\chi$, and by linear algebra~$X^{md}\chi(X^{-1})$ and~$\chi$ are equal up to multiplication by an element of~$F^\times$. Thus~$P_1$ and~$P_2$ have the same multiplicity in~$\chi$,
 and therefore~$\chi(0)=1$, taking~$a_0=P_1(0)=\frac{1}{b_0}$ into account.

Case~$\mu_{g,F}=P^{\nu}$: As in the case above we obtain that~$X^{\deg P}P(X^{-1})\frac{1}{P(0)}$ and~$P$ coincide. We split~$P$ in an algebraic closure of~$F$: 
\[P=(X-\lambda_1)\ldots (X-\lambda_l).\]
Then inverse map of~$F^\times$ permutes the roots of~$P$. Thus~$P(0)=1$ if every root of~$P$ satisfies~$\lambda^{-1}\neq\lambda$. If the latter is not the case then~$P$ has a root~$\lambda$ which is~$1$ or~$-1$,
and in this case~$P=(X-\lambda)$ because it is irreducible. Then~$\chi$ is an even power of~$(X-\lambda)$ and thus~$\chi(0)=1.$ 
\end{proof}

\section{Stevens' cohomology argument on double cosets}
Let~$p$ and~$l$ be different primes and let~$\Gamma$ be an~$l$-group acting continuously on some topological Hausdorff group~$Q$. Kurinczuk 
and Stevens proved the following result in~\cite[2.7]{KurinczukStevens:15}: We denote by~$S^\Gamma$ the set of~$\Gamma$-fixed points of~$S$ 
for any subset~$S$ of~$Q$.

\begin{theorem}[\cite{KurinczukStevens:15}~2.7(ii)(b)]\label{thmDC}
 Suppose~$U_1$ and~$U_2$ are two~$\Gamma$-stable pro-$p$-subgroups of~$Q$. Let~$H$ be a further~$\Gamma$-stable subgroup of~$Q$ such that for every~$h\in H$ the following identity holds:
 \begin{equation}\label{eqDCCShaun}
 (U_1hU_2)\cap H=(U_1\cap H)h(U_2\cap H).
\end{equation}
 Then, we have the double coset decomposition:
 \[(U_1HU_2)^\Gamma=U_1^\Gamma H^\Gamma U_2^\Gamma.\]
\end{theorem}

In this section we are going to generalize this result in allowing~$\Gamma$-stable cosets of~$H$ instead of~$H$. 

\begin{proposition}\label{propDCnew}
 Suppose~$U_1$ and~$U_2$ are two~$\Gamma$-stable pro-$p$-subgroups of~$Q$ and~$H$ is a further~$\Gamma$-stable subgroup of~$Q$. 
 Suppose~$gH$ is a~$\Gamma$-stable coset of~$H$ in~$Q$ such that for every~$h\in H$ the following identity holds:
 \begin{equation}\label{eqDCCnew}
 (U_1ghU_2)\cap gH=(U_1\cap gHg^{-1})gh(U_2\cap H).
   \end{equation} 
  Then, we have:
  \[(U_1gHU_2)^\Gamma=U_1^\Gamma (gH)^\Gamma U_2^\Gamma.\]
\end{proposition}

\begin{remark}\label{remDoubleCoset}
 The condition~\eqref{eqDCCnew} for~$U_1,g,H,U_2$ is equivalent to~\eqref{eqDCCShaun}
 for~$g^{-1}U_1g,H,U_2$. So it is enough to establish~\eqref{eqDCCShaun} for a big class of triples~$U'_1,H',U'_2$. 
\end{remark}

The proof is literally the same, but we repeat the argument where minor changes occur. At first we need the following two statements
from~\cite{KurinczukStevens:15}.

\begin{lemma}[\cite{KurinczukStevens:15}~2.7(i)]\label{lemKurStevDC2p7i}
 Suppose that~$U_1$ and~$U_2$ are two subgroups of~$Q$ such that the (non-abelian) cohomology~$H^1(\Gamma,gU_1g^{-1}\cap U_1)$ is trivial. Then we have for any~$\Gamma$-fixed element~$g$ of~$Q$ the identity~$(U_1gU_2)^\Gamma=U_1^\Gamma gU_2^\Gamma$.
\end{lemma}

\begin{lemma}[\cite{KurinczukStevens:15}~2.7(ii)(a)]\label{lemKurStevDC2p7iia}
 For any two~$\Gamma$-stable pro-$p$-subgroups~$U_1$ and~$U_2$ of~$Q$ and any element~$g$ of~$Q$ the following assertions are equivalent:
 \begin{enumerate}
  \item $(U_1gU_2)^\Gamma\neq \emptyset$
  \item $U_1gU_2$ is~$\Gamma$-stable.
 \end{enumerate}
\end{lemma}

\begin{proof}[Proof of Proposition~\ref{propDCnew}]
 There is nothing to prove for the inclusion~$\supseteq$, so we continue with the other inclusion. Let~$x$ be an element of~$(U_1gHU_2)^\Gamma$. 
 Then there is an element~$h$ of~$H$ and there are elements~$u_1\in U_1$,~$u_2\in U_2$ such that~$x=u_1ghu_2$. We have to show that we 
 could have taken~$h$ such that~$gh$ is~$\Gamma$-fixed, because Lemma~\ref{lemKurStevDC2p7i} would then imply that~$x$ is an element 
 of~$U_1^\Gamma (gH)^\Gamma U_2^\Gamma$. 
 Now,~$U_1ghU_2$ and~$gH$ are~$\Gamma$-stable, so, by~\eqref{eqDCCShaun} and Lemma~\ref{lemKurStevDC2p7iia}, there is a~$\Gamma$-fixed 
 point in~$(U_1\cap gHg^{-1})gh(U_2\cap H)$, which we could have chosen instead of~$gh$ in the product decomposition of~$x$. This finishes the proof.  
 \end{proof}

\section{Strata for~$G$}\label{secStrataAndCharactersForG}

In this section we generalize the notion of self-duality for strata from~$p$-adic classical groups, see~\cite{stevens:05} 
and~\cite{skodlerackStevens:16}, to its quaternionic inner forms. We generalize the usual statements about strata,
 as for example about the intertwining formula and a Skolem--Noether result, and recall endo--classes of semisimple strata.
We take all definitions and results from~\cite{skodlerack:17-1}. 

\subsection{First definitions}

We refer to~\cite{stevens:05} and~\cite{skodlerackStevens:16} for the non-quaternionic case. 
We use all notation and definitions in~\cite{skodlerack:17-1}. 
For example the definitions for strata (pure, simple, semi-pure, semisimple) can be found in section 4. 

A stratum is denoted as a quadruple~$[\Lambda,n,r,\beta]$ and if we write 
$\Delta'$ for a stratum, then the entries appear with a superscript~$'$, i.e.~$[\Lambda',n',r',\beta']$, and similar with 
subscripts:~$\Delta_c:=[\Lambda^c,n_c,r_c,\beta_c]$ (The superscript on~$\Lambda$ is not a typo.) 
We write~$E$ for~$F[\beta]$ and similar~$E'$,~$E_c$ etc.. And we write~$\C_?(!)$ for the centralizer of~$!$ in~$?$. 

If~$\Delta$ is a semi-pure stratum it has an associated splitting which we denote by~$V=\oplus_{i\in I}V^i$ coming from the 
decomposition of~$E$ into a product of fields, and we call the corresponding idempotents~$1^i$. Split means that~$\Delta$ is the
direct sum of its restrictions~$\Delta|_{V^i}$ which is called the~$i$th block of~$\Delta$. 

Given a full~$o_F$-module~$M$ in~$V$ we define the dual for~$M$ with respect to~$h$ via
\[M^\#=\{v\in V\mid h(v,M)\subseteq \mf{p}_D\}.\] 
The form~$h$ defines an involution~$\#$ on the set of lattice sequences for~$V$, in defining~$\Lambda^{\#}_i$ as~$(\Lambda_{-i})^{\#}$. 
This depends on the choice of~$h$. 
A lattice sequence~$\Lambda$ is called self-dual if~$\Lambda^\#$ and~$\Lambda$ differ by a translation, i.e. there is an 
integer~$k$ such that~$\Lambda-k$, which is defined as~$(\Lambda_{i+k})_{i\in\ZZ}$, coincides with~$\Lambda^\#$. 
This gives an action of~$\#$ on the set of strata for~$V$:
\[\Delta^\#:=[\Lambda^\#,n,r,-\sigma_h(\beta)].\]
In our notation the latter says:~$n^\#=n,~r^\#=r$ and~$\beta^\#=-\sigma_h(\beta)$. 

\begin{definition}
 A stratum~$\Delta$ is called~\emph{self-dual} if~$\Delta^\#$ and~$\Delta$ only differ by a translation of~$\Lambda$, i.e. if~$\Lambda$ is self-dual and~$\sigma_h(\beta)=-\beta$. Further if~$\Delta$ is a self-dual semi-pure stratum, 
 then~$\sigma$ induces an involution on the index set~$I$, which we also call~$\sigma$. 
 The action of~$\{1,\sigma\}$ decomposes the index set~$I$ into
a set of fixed points~$I_0$ and the set~$I_{+-}$ of elements which have an orbit of length~$2$. 
We usually choose a section~$I_+\subseteq I_{+-}$ through all orbits of length~$2$. 
Given a union of~$\sigma$-orbits~$J\subseteq I$ we denote the restriction of~$h$ to~$V^J:=\oplus_{i\in J}V^i$ by~$h_J$.
We also  write~$h_{i,\ldots,j}$ instead of~$h_{\{i,\ldots,j\}}$.
 We call~$\Delta$~\emph{skew} if~$I=I_0$.
\end{definition}


\subsection{Diagonalization for self-dual strata}

One of the first important properties for strata is the diagonalization proposition for self-dual simple strata.
Let us first state the diagonalization proposition for~$\GL_D(V)$.

\begin{proposition}[\cite{skodlerack:17-1}~Theorem 4.30]\label{propDiagonalizationStrataGL}
 Let~$V^i,\ i=1,\ldots,l$, be sub-$D$-vector spaces of~$V$ whose direct sum 
 is~$V$. Let~$\Delta$ be a stratum which splits under~$\oplus_j V^j$ into a direct sum of pure strata~$\Delta|_{V^j}$.
 Suppose further that~$\Delta$ is equivalent to a simple strata. 
 Then, there is a simple stratum which is equivalent to~$\Delta$ and split by~$\oplus_iV^i$. 
\end{proposition}

We want to prove: 

\begin{proposition}[see \cite{skodlerackStevens:16}~6.6 for the case over~$F$]\label{propDiagonalizationStrataG}
 Let~$V=\bigoplus_{j\in J}V^j$ a decomposition of~$V$ such that the projections~$1^j:V\ra V^j$ are permuted by~$\sigma_h$. 
 Let~$\Delta$ be a self-dual stratum which is split under~$(V^j)$ such that the restrictions~$\Delta|_{V^j}$ 
 and~$\Delta$ are equivalent to a simple stratum. Then there is a self-dual simple stratum which is split by~$(V^j)$ and equivalent
 to~$\Delta$.
\end{proposition}

For the non-quaternionic skew case, see~\cite[Theorem 6.16]{skodlerackStevens:16}.
For the proof we need three further technical lemmas: 

\begin{lemma}[\cite{skodlerack:17-1}~4.28]\label{lemAquivToSimpleStratum}
 Let~$\Delta$ be a stratum. Then~$\Delta$ is equivalent to a simple stratum if and only if~$\Res_F(\Delta)$ is equivalent to a simple 
 stratum.
\end{lemma}

\begin{lemma}[\cite{skodlerack:17-1}~4.21]\label{lemPureSimpleResF}
 Let~$\Delta$ be a pure stratum. Then~$\Delta$ is a simple stratum if and only if~$\Res_F(\Delta)$ is a simple stratum.  
\end{lemma}

\begin{lemma}[\cite{stevens:01-2}~1.9]\label{lemLimitOfSinmpleStrata}
 Let~$[\Lambda_F,n,r,\gamma_t]$ be a sequence of equivalent simple strata in~$\End_F(V)$ such that~$\gamma_t$ converges to some~$\gamma$
 in~$\End_F(V)$. Then the stratum~$[\Lambda_F,n,r,\gamma]$ is simple.
\end{lemma}

Here we consider~$h_F:=\trd_{D|F}\circ h$. 

\begin{proof}[Proof of Proposition~\ref{propDiagonalizationStrataG}]
 This proof uses the strategy of~\cite[1.10]{stevens:01-2}, where the author shows the existence of a sequence~$\gamma_t,\ t\in \bbN_0$,
 satisfying:
 \begin{itemize} 
  \item[(i)] $[\Lambda_F,n,r,\gamma_t]$ is simple and equivalent to~$\Res_F(\Delta)$ for all~$t\in\bbN_0$,
  \item[(ii)] $\gamma_t+\sigma_{h_F}(\gamma_t)\in\mf{a}_{\Lambda_F,-r+t}$ for all~$t\in\bbN_0$.
  \item[(iii)] $\gamma_t-\gamma_{t+1}\in\mf{a}_{\Lambda_F,-r+t}$ for all~$t\in\bbN_0$. 
 \end{itemize}
 We show that we could have chosen~$\gamma_t$ in~$\prod_{i=1}^l\End_DV^i$. 
 This follows from Proposition~\ref{propDiagonalizationStrataGL} for~$\gamma_0$. So suppose that 
 $\gamma_0,\ldots,\gamma_t$ are elements of~$\prod_{i=1}^l\End_DV^i$ which satisfy (i)-(iii).~$\sigma$ induces an involution on~$J$. Let~$J_0$ be the set of~$\sigma$-fixed points in~$J$.
 We denote the stratum~$[\Lambda_F,n,r-t-1,\frac{\gamma_t-\sigma_{h}(\gamma_t)}{2}]$ by~$\ti{\Delta}$.
 Let~$s$ be a tame corestriction with respect to~$\gamma$. The derived stratum~$\partial_{\gamma_t}(\Res_F(\ti{\Delta}))$
 is equivalent to a stratum~$[\Lambda,r-t,r-t-1,\delta]$ for some~$\delta\in F[\gamma_t]$, by the proof of~\cite[1.10]{stevens:01-2}. 
 Now,~$\ti{\beta}$ and~$\delta$ commute with the projections~$1^j$ and thus~$\partial_{1^j\gamma_t}(\Res_F(\ti{\Delta}|_{V^j}))$ is 
 equivalent to a simple stratum. Therefore the strata~$\Res_F(\ti{\Delta})$ and~$\Res_F(\ti{\Delta}|_{V^j})$ are equivalent to simple 
 strata, by~\cite[2.13]{secherreStevensVI:12}. 
 Thus~$\ti{\Delta}$ and its restrictions are equivalent to simple strata, by Lemma~\ref{lemAquivToSimpleStratum}. 
 Now, we can find by Proposition~\ref{propDiagonalizationStrataGL} a simple stratum~$\ti{\ti{\Delta}}$ split under~$(V^j)$ and 
 equivalent to~$\ti{\Delta}$. We put~$\gamma_{t+1}:=\ti{\ti{\beta}}$. 
 
 The sequence~$(\gamma_t)_{t\in\bbN_0}$ converges and we denote the limit by~$\gamma$.  Then~$[\Lambda_F,n,r,\gamma]$ is simple, by 
 Lemma~\ref{lemLimitOfSinmpleStrata}, and therefore~$[\Lambda,n,r,\gamma]$ is simple, 
 by Lemma~\ref{lemPureSimpleResF}. This finishes the proof. 
\end{proof}

One consequence is:
\begin{corollary}\label{corSemisimpleEquiv}
 Let~$\Delta$ be a semisimple stratum such that~$\#\Delta$ is equivalent to~$\Delta$. Then,~$\Delta$ is equivalent to a self-dual 
 semisimple stratum. 
\end{corollary}

\begin{proof}
 By Corollary~\cite[4.37]{skodlerack:17-1} and Corollary~\cite[A.4]{kurinczukSkodlerackStevens:16} we can assume that the idempotents of the associate splitting of~$\Delta$ are permuted
 by the action of~$\sigma_h$. We apply Proposition~\ref{propDiagonalizationStrataG} on~$\Delta_i$, for~$i\in I_0$, and we replace~$\beta_i$
 by~$-\sigma_h(\beta_{\sigma(i)})$ for~$i\in I_-$. 
\end{proof}

\subsection{Intersection formulas}

For the next sections we need some special cases of~\eqref{eqDCCShaun}. 
For this we need to recall that on attaches to a semisimple stratum two further pro-$p$-subgroups of~$\ti{G}$:
$S(\Delta)$ and~$M(\Delta):=1+\mf{m}(\Delta)$, see~\cite[before 4.25, after 5.14]{skodlerack:17-1} for the definitions of~$S(\Delta)$ 
and~$\mf{m}(\Delta)$. We need the technique of~$\ddag$ construction for strata, see~\cite[4.6]{skodlerack:17-1}, to attach to a 
stratum~$\Delta$ a stratum~$\Delta^\ddag$ where the lattice sequence is principal.  

\begin{proposition}\label{propIntersectionFormulas}
 Let~$\Delta_1$ and~$\Delta_2$ be two semisimple strata which share the associated splitting and the parameters~$r$ and~$n$. 
 Let~$H$ be one of the following groups: 
 \begin{itemize} 
  \item $\prod_{i\in I_1}\Aut_DV^i$ or
  \item $C_A(\beta_1)^\times$, if~$\beta_1=\beta_2$. 
 \end{itemize}
Let~$h$ be an element of~$H$. Then we have the following 
intersection formulas.
 \begin{enumerate}
  \item \label{propDCCForS} 
  $(S(\Delta_2)h S(\Delta_1))\cap H =(S(\Delta_2)\cap H)h(S(\Delta_1)\cap H)$,
  \item \label{propDCCForM} 
  $(M(\Delta_2)h M(\Delta_1))\cap H =(M(\Delta_2)\cap H)h(M(\Delta_1)\cap H)$,
 \end{enumerate}
\end{proposition}

\begin{proof}
 We only show~\ref{propDCCForS} because the proof of the second equation is obtained in replacing~$S$ by~$M$. 
 We start with the case that~$H$ is equal to~$\prod_{i\in I_1}\Aut_DV^i$. Consider the group~$\Gamma:=\{\pm 1\}^{\# I_1}$ acting 
 on~$\ti{G}$ by conjugation. Its fixed point set is~$H$ and~$\Gamma$ is contained in~$P(\Lambda^1)\cap P(\Lambda^2)$.
 Lemma~\ref{lemKurStevDC2p7i} implies~\ref{propDCCForS}. 
 Let us now assume~$\beta_1=\beta_2=:\beta$ and that~$H$ is~$C_A(\beta)^\times$. 
 Without loss of generality we can assume that both lattice sequences have the same~$F$-period. We make the~$\ddag$-construction 
 for~$\Delta_1\otimes L$ and~$\Delta_2\otimes L$, and we get~$(\Delta_1\otimes L)^\ddag$ and~$(\Delta_2\otimes L)^\ddag$ 
 which are in fact equal to~$\Delta_1^\ddag\otimes L$ and~$\Delta_2^\ddag\otimes L$, respectively. Both latter strata are 
 semisimple and the lattice sequences on the blocks are principal lattice chains. Thus there is an invertible element~$g$ of
 $C:=C_{\End_L(V^\ddag)}(\beta_1^\ddag)$ which sends~$\Lambda^{1,\ddag}$ to~$\Lambda^{2,\ddag}$.
 Formula~\ref{propDCCForS} is true for the triple~$(\Delta_2\otimes L)^\ddag, (\Delta_2\otimes L)^\ddag$,~$C^\times$, and 
 for every element~$c\in C^\times$ by~\cite[(1.6.1)]{bushnellKutzko:93}. We conjugate back with~$g$ to obtain~\ref{propDCCForS} 
 for~$(\Delta_2\otimes L)^\ddag, (\Delta_1\otimes L)^\ddag$,~$C^\times$ and every~$c\in C^\times$. In particular the formula is valid for 
 $\diag(b_i|\ i\in I_1)$ with~$b_i\in C_{\Aut_L(V)}(\beta_1)^\times$,~$i\in I_1$. Using a limit-argument 
 we obtain for all~$h\in C_{\Aut_L(V)}(\beta)^\times$
 \[(S(\Delta_2\otimes L)h S(\Delta_1\otimes L))\cap C_{\Aut_L(V)}(\beta_1)^\times \subseteq 
 (S(\Delta_2\otimes L)\cap C_{\Aut_L(V)}(\beta_1)^\times)h (S(\Delta_1\otimes L)\cap C_{\Aut_L(V)}(\beta_1)^\times).\]
 The last inclusion is an equality. We consider~$h\in C_A(\beta_1)^\times$ and take~$\tau$-fixed 
 points. Then~\ref{lemKurStevDC2p7i} implies the assertion. 
\end{proof}

\subsection{Matching and intertwining}

The intersection formulas of the last section allow us to determine the $G$-intertwining of two self-dual semisimple strata with the same~$\beta$. Two semisimple strata with same parameters which intertwine by an element of~$\ti{G}$ 
possess a matching, i.e. they have a unique one-to-one correspondence between the block structures, in particular with coinciding dimensions. We prove in the self-dual case that this correspondence is given by isometries for~$\sigma$-fixed blocks.

In this section we suppose that~$\Delta$ and~$\Delta'$ are two semisimple strata which satisfy~$e(\Lambda|F)=e(\Lambda'|F)$,
$n=n'$ and~$r=r'$.

At first we recall the matching and the intertwining formulas for semisimple strata.

\begin{theorem}[\cite{skodlerack:17-1}~4.41 (with~4.32)]\label{thmMatchingStrataGL}
 Suppose~$I(\Delta,\Delta')\neq \emptyset$. Then there is a unique bijection~$\zeta:\ I\ra I'$ such that
 the set~$I(\Delta,\Delta')\cap \prod_{i\in I}\Hom_D(V^i,V^{\zeta(i)})$ is not empty. The latter non-emptiness condition is
 equivalent in saying that for all~$i\in I$ the~$D$-dimensions of~$V^i$ and~$V^{\zeta(i)}$ coincide and the direct 
 sum~$\Delta_i\oplus\Delta'_{\zeta(i)}$ is equivalent to a simple stratum. 
%
\end{theorem}

The map~$\zeta$ is called the matching of~$(\Delta,\Delta')$, and we denote it also as~$\zeta_{\Delta.\Delta'}$. 

\begin{theorem}[\cite{skodlerack:17-1}~4.36]\label{thmIntertwiningStrataGL}
 Suppose~$I(\Delta,\Delta')\neq \emptyset$ with matching~$\zeta$. Then the following holds. 
 \begin{enumerate}
  \item $I(\Delta,\Delta')= M(\Delta')(I(\Delta,\Delta')\cap \prod_{i\in I}\Hom_D(V^i,V^{\zeta(i)})) M(\Delta)$.
  \item Suppose there is an element~$\ti{g}$ of~$\ti{G}$ such that~$\ti{g}\beta\ti{g}^{-1}=\beta'$, then we have
  \[I(\Delta,\Delta')= M(\Delta')\ti{g}C_A(\beta)^\times M(\Delta).\]
 \end{enumerate}
\end{theorem}

\begin{proof}
 The second assertion follows from~\cite[4.36]{skodlerack:17-1} and the first assertion follows from the second by 
 diagonalization, Proposition~\ref{propDiagonalizationStrataGL}, and the fact
 that~$M(*)$ only depends on the equivalence class of the semisimple stratum. 
\end{proof}

We now can come to the self-dual case: 

\begin{theorem}[see~\cite{skodlerackStevens:16}~6.22 for the case over~$F$]\label{thmIntertwiningStrataG}
Suppose~$\Delta$ and~$\Delta'$ are intertwined by an element of~$G$ with matching~$\zeta$. Then the following holds.
\begin{enumerate}
\item\label{thmIntertwiningStrataG.i} $I_G(\Delta,\Delta')=(G\cap M(\Delta'))(I_G(\Delta,\Delta')\cap \prod_{i\in I}\Hom_D(V^i,V^{\zeta(i)}))(G\cap M(\Delta))$.
\item Suppose there is an element~$\ti{g}$ of~$\ti{G}$ such that~$\ti{g}\beta\ti{g}^{-1}=\beta'$, then we have
\[I_G(\Delta,\Delta')=(G\cap M(\Delta'))(G\cap\ti{g}C_A(\beta)^\times)(G\cap M(\Delta)).\]
\end{enumerate}
\end{theorem}
\begin{proof}
 This Theorem follows from Theorem~\ref{thmIntertwiningStrataGL}, Proposition~\ref{propIntersectionFormulas} and 
 Proposition~\ref{propDCnew} for the group~$\Gamma=\{\sigma,1\}$. 
\end{proof}

The last Theorem has now consequences for the matching. 

\begin{corollary}[see~\cite{kurinczukSkodlerackStevens:16}~9.5 for the case over~$F$]\label{corIsometriesForMatchings}
Suppose~$\Delta$ and~$\Delta'$ are self-dual and intertwined by an element of~$G$ with matching~$\zeta$. Then~$I_G(\Delta,\Delta')\cap \prod_{i\in I}\Hom_D(V^i,V^{\zeta(i)})$ is non-empty and~$h_J$ is isometric to~$h_{\zeta(J)}$ for all~$\sigma$-invariant subsets~$J$ of~$I$. In particular~$\zeta$ and~$\sigma$ commute. 

\end{corollary}

\begin{proof}
 The non-emptiness of~$I_G(\Delta,\Delta')\cap \prod_{i\in I}\Hom_D(V^i,V^{\zeta(i)})$  follows from~\ref{thmIntertwiningStrataG}\ref{thmIntertwiningStrataG.i}.
 Every element of this intersection restricts to an isometry from~$h_J$ to~$h_{\zeta(J)}$.  If~$g$ is an element of the above intersection then~$\sigma(g)$ too, and the uniqueness of the matching implies 
 that~$\zeta$ and~$\sigma$ commute. 
\end{proof}

Later we are going to see that the non-emptiness of~$I_G(\Delta,\Delta')$ is not needed in Corollary~\ref{corIsometriesForMatchings}
for~$\sigma$ and~$\zeta$ to commute.

\subsection{Skolem-Noether} 
We prove a version of Skolem--Noether for~$G$.  

\begin{theorem}[see~\cite{skodlerackStevens:16}~5.2 for the case over~$F$]\label{thmSkolemNoetherG}
 Let~$\Delta$ and~$\Delta'$ be two pure strata such that~$e(\Lambda|F)=e(\Lambda'|F)$ and~$n=n'>r\geq r'$. Suppose that there is an element of~$G$  which intertwines~$\Delta$ with~$\Delta'$ and that~$\beta$ and~$\beta'$ have the same minimal polynomial. Then the element~$\beta$ is  conjugate to~$\beta'$ by an element of~$G$. 
\end{theorem}

\begin{proof}
 We only consider pure strata, so without loss of generality we can assume~$n=r-1=r'-1$.
 We can find an element~$\ti{g}\in\ti{G}$ such that~$\ti{g}\beta\ti{g}^{-1}=\beta'$. The proof consists of two steps. 
 \begin{enumerate}
  \item[Step 1] Let us at first assume that both strata are simple. 
  The set~$G\cap\ti{g}B$ is non-empty by Theorem~\ref{thmIntertwiningStrataG} because~$I_G(\Delta,\Delta')$ is non-empty. 
  This finishes Step 1. 
  \item[Step 2] There are simple strata~$\ti{\Delta}$ and~$\ti{\Delta}'$ equivalent to~$\Delta$ and~$\Delta'$, respectively.
  By Proposition~\ref{propDiagonalizationStrataG} we can assume that~$\ti{\beta}$ and~$\ti{\beta}'$ have the same minimal polynomial.
  Let~$e_w$ be the wild ramification index of~$\ti{E}|F$. There is an element~$\gamma\in \ti{E}$ which is congruent 
  to~$\tilde{\beta}^{e_w}$ and generates the maximal tamely ramified sub-extension~$\ti{E}_{tr}$ in~$\ti{E}$. 
  By~\cite[2.15]{kurinczukSkodlerackStevens:16} there is an~$F$-linear field monomorphism~$\phi:\ti{E}_{tr}\rightarrow E$ 
  such that~$\phi(\gamma)$ is congruent to~$\beta^{e_w}$ in~$E$. We take~$\gamma'$ to be the image of~$\gamma$ under
  the isomorphism~$\ti{E}\cong \ti{E}'$ which sends~$\ti{\beta}$ to~$\ti{\beta}'$. We repeat the argument to get an~$F$-linear field monomorphism~$\phi':\ti{E}'_{tr}\ra E'$ such that~$\phi'(\gamma')$ is congruent to~$\beta'^{e_w}$.
  The elements~$\phi(\gamma)$ and~$\phi'(\gamma')$ are conjugate  by an element of~$G$ by Step 1, and we can assume that~$\phi(\gamma)$ 
  and~$\phi'(\gamma')$ coincide and we denote~$\phi(\gamma)$ by~$t$. If~$E|F[t]$ and~$E'|F[t]$ are isomorphic by a 
  map which sends~$\beta$ to~$\beta'$,   then~$\beta$ is conjugate to~$\beta'$ by an element of~$C_G(F[t])$ 
  by~\cite[5.1]{skodlerackStevens:16}. 
  Let~$\psi$ be the~$F$-linear isomorphism from~$E$ to~$E'$ which sends~$\beta$ to~$\beta'$. Then
  \begin{equation}\label{eqCentr}\psi(t)\equiv\psi(\beta^{e_w})=\beta'^{e_w}\equiv t.\end{equation}
  Thus,~$\psi(x)$ is congruent to~$x$ in~$E'$ for all~$x\in F[t]$ because~$t$ is minimal over~$F$.
  We obtain at first that~$\psi$ is the identity on the maximal unramified sub-extension~$F[t]_{ur}$ of~$F[t]|F$. Let~$\pi_t$ 
  be a uniformizer of~$F[t]$ whose~$e(F[t]|F)$th power is an element of~$F[t]_{ur}$. We just denote the ramification index of~$F[t]|F$ 
  by~$e_t$. From the congruence~$\psi(\pi_t)\equiv \pi_t$ we obtain an element~$y\in 1+\mf{p}_{E'}$ such that~$y\pi_t=\psi(\pi_t)$. 
  Thus~$y^{e_t}=1$. Therefore~$y$ and~$1$ are~$e_t$th root of unity which are congruent in~$E'$. Thus~$y=1$ because~$e_t$ is not divisible by the residue characteristic of~$F$. We have established that~$\psi$ is the identity on~$F[t]$ which finishes the proof. \end{enumerate}
\end{proof}

This has an immediate analogue consequence for semisimple strata.
\begin{corollary}\label{corSkolemNoetherSemisimple}
Suppose~$\Delta$ and~$\Delta'$ are self-dual semisimple strata with~$r=r'$, the same~$F$-period  and a bijection~$\zeta:I\ra I'$ such that~$\beta_i$ and~$\beta'_{\zeta(i)}$ have the same minimal polynomial over~$F$. 
Suppose further  there is an element of~$G\cap\prod_{i\in I}\Hom(V^i,V^{\zeta(i)})$  which intertwines~$[\Lambda^i,n_i,n_{i}-1,\beta_i]$ with~$[\Lambda'^{\zeta(i)},n'_{\zeta(i)},n'_{\zeta(i)}-1,\beta'_{\zeta(i)}]$ for all~$i\in I$ with~$\beta_i\neq 0$. Then~$\beta$ is conjugate to~$\beta'$ by an element of~$G$. 
\end{corollary}
 
\begin{proof}
 We take at first an index~$i\in I_0$ with~$\beta_i\neq 0$. 
 Then we get~$n_i=n'_{\zeta(i)}$ from~\cite[6.9]{skodlerackStevens:16}, and we get  by~\ref{thmSkolemNoetherG} an isometry from~$h_i$ to~$h_{\zeta(i)}$ which conjugates~$\beta_i$ to~$\beta_{\zeta(i)}$. 
 For~$i\in I_+$ we can take any~$D$-linear isomorphism~$g_i$ from~$V^i$ to~$V^{\zeta(i)}$ which conjugates~$\beta_i$ 
 to~$\beta_{\zeta(i)}$, and we obtain an element~$\sigma_h(g_i)\in\Hom(V^{\zeta(\sigma(i))},V^{\sigma(i)})$ 
 which conjugates~$\beta'_{\zeta(\sigma(i))}$ to~$\beta_{\sigma(i)}$. This finishes the proof. 
\end{proof}

\subsection{Conjugate semisimple strata}
In this section we prove an ``intertwining implies conjugacy''-kind of result for self-dual semisimple strata. It will lead 
directly to the analogue result for self-dual semisimple characters.
In this section we fix two intertwining semisimple strata~$\Delta$ and~$\Delta'$ and we assume~$n=n'$,~$r=r'$ and~$e(\Lambda|F)=e(\Lambda'|F)$.

Intertwining does not imply conjugacy in general, even if one would assume that~$\ti{g}\Lambda^i$ is equal to~$\Lambda^{\zeta(i)}$ 
for some~$\ti{g}\in\ti{G}$ because one has to keep track of the embedding type. This is the reason why we have 
introduced the map~$\bar{\zeta}$ in~\cite[4.46]{skodlerack:17-1}. It is the map from~$\kappa_E$ to~$\kappa_{E'}$ given by the 
diagram
\begin{equation}\label{eqZetqbar}
 \kappa_E\hookrightarrow \mf{a}_0/\mf{a}_1\rightarrow (g\mf{a}_0g^{-1}+\mf{a}'_0)/(g\mf{a}_1g^{-1}+\mf{a}'_1)
 \leftarrow \mf{a}'_0/\mf{a}'_1\hookleftarrow \kappa_{E'},\ g\in I(\Delta,\Delta'). 
\end{equation}
This map does not depend on the choice of~$g$. We call the pair~$(\zeta,\bar{\zeta})$ the matching pair. 
In fact we do not need the whole map for the conjugacy. The restriction to~$\kappa_{E_D}$  is enough.
The algebra~$E_D$ is the product of the unramified field extensions~$E_{i,D}|F$ of degree~$\gcd(\sqrt{[D:F]},f(E_i|F))$,~$i\in I$.

\begin{theorem}[\cite{skodlerack:17-1}~4.48]\label{thmIntConStrataGL}
 Suppose there is an element~$\ti{g}\in\prod_i\Hom_D(V^i,V^{\zeta(i)})$ such that~$\ti{g}\Lambda=\Lambda'$ and such that 
 the conjugation with~$\ti{g}$ verifies~$\bar{\zeta}|_{\kappa_{E_D}}$.
 Then there is an element~$u$ of~$\prod_i\Hom_D(V^i,V^{\zeta(i)})$ such that~$u\Lambda=\Lambda'$ and~$u\Delta$ is equivalent to
 $\Delta'$. We can choose~$u$ to further satisfy~$u\beta u^{-1}=\beta'$ if~$\beta$ and~$\beta'$ have the same characteristic polynomial  
 over~$F$. 
\end{theorem}

The second part of this Theorem follows directly from the proof of~\cite[4.48]{skodlerack:17-1}. We continue now with its self-dual 
analogue. 

\begin{theorem} \label{thmIntConStrataG}
 Suppose both strata are self-dual and intertwine by an element of~$G$.
 Let~$g$ be an element of~$G\cap\prod_i\Hom_D(V^i,V^{\zeta(i)})$ which satisfies~$g\Lambda=\Lambda'$ and such that 
 the conjugation with~$g$ verifies~$\bar{\zeta}|_{\kappa_{E_D}}$.
 Then there is an element~$u$ of~$G\cap \prod_i\Hom_D(V^i,V^{\zeta(i)})$ such that~$u\Lambda=\Lambda'$ and~$u\Delta$ is equivalent to
 $\Delta'$. We can choose~$u$ to further satisfy~$u\beta u^{-1}=\beta'$ if~$\beta$ and~$\beta'$ have the same characteristic 
 polynomial  over~$F$. 
\end{theorem}

\begin{proof}
 We apply~$g$ and can assume~$\Lambda^i=\Lambda'^{\zeta(i)}$, for all~$i$, and~$\bar{\zeta}|_{\kappa_{E_D}}=\id|_{\kappa_{E_D}}$ 
 without loss of generality. We can therefore restrict to the case of a single~$\sigma$-orbit, i.e.~$I$ is either a singleton or 
 consists of two non-fixed points of~$\sigma$. We apply in the latter case Theorem~\ref{thmIntConStrataGL} for one index~$i$, to get 
 an appropriate element~$u_i$, and~$u:=\diag(u_i,\sigma_h(u_i)^{-1})$ satisfies the assertions. So we are left with the singleton 
 case~$I=\{i_0\}$, i.e. we assume that both strata are simple. By diagonalization, see~\ref{propDiagonalizationStrataG},
 we can assume that~$\beta$ and~$\beta'$ have the same minimal polynomial over~$F$. The element~$\beta$ is conjugate to~$\beta'$ by 
 some element~$g_0$ of~$G$ by Theorem~\ref{thmSkolemNoetherG}. The conjugation by~$g_0$ 
 induces~$\bar{\zeta}|_{\kappa_{E_D}}=\id_{\kappa_{E_D}}$ because~$g_0$ intertwines~$\Delta$ with~$\Delta'$. 
 Thus,~$g_0xg_0^{-1}+\mf{a}_1$ is equal to~$x+\mf{a}_1$ in~$\mf{a}_0/\mf{a}_1$ for all~$x\in o_E$. By~\cite[4.39]{skodlerack:17-1} there is an 
 element~$v$ of~$P(\Lambda)$ which conjugates~$\beta$ to~$\beta'$. We consider the~$C_{\ti{G}}(E_D)$-equivariant 
 affine Broussous--Lemaire isomorphism 
 \[j_{E_D}:\ B_{red}(\ti{G})^{E_D^\times}\DistTo B_{red}(C_{\ti{G}}(E_D))\]
 between the reduced Bruhat--Tits buildings, see~\cite{broussousLemaire:02}.
 The translation classes~$[\Lambda]$ and~$[g_0^{-1}\Lambda]$ are elements of~$B_{red}(\ti{G})^{E_D^\times}$. They satisfy
 \[v^{-1}g_0j_{E_D}([g_0^{-1}\Lambda])=j_{E_D}(v^{-1}g_0[g_0^{-1}\Lambda])=j_{E_D}([\Lambda]).\]
 The element~$v^{-1}g_0$ acts type-preserving on~$B_{red}(C_{\ti{G}}(E_D))$ because~$\nu_{E_D}(\Nrd_{E_D}(v^{-1}g_0))$ 
 vanishes, i.e. the points~$j_{E_D}([g_0^{-1}\Lambda])$ and~$j_{E_D}([\Lambda])$ have the same simplicial
 type in~$B_{red}(C_{\ti{G}}(E_D))$. Note that both are points of the building~$B(C_G(E_D))$ of the centralizer of~$E_D$ in~$G$. 
 By~\cite[5.2]{skodlerack:13} 
 there is an element~$g_1\in C_G(E_D)$ such that 
 \[g_1j_{E_D}([g_0^{-1}\Lambda])=j_{E_D}([\Lambda]).\]
 The~$C_G(E_D)$-equivariance and the injectivity of~$j_{E_D}$ imply~$[g_1g_0^{-1}\Lambda]=[\Lambda]$. 
 Thus~$g_1g_0^{-1}\Lambda$ is a translate of~$\Lambda$, i.e.~$g_1g_0^{-1}$ lies  in the normalizer of~$\Lambda$ 
 and hence in~$\mf{a}_0^\times$, because it is an element of~$G$. The element~$u:=g_1g^{-1}_0$ is an element of~$P(\Lambda)\cap G$ which
 conjugates~$\beta$ to~$\beta'$. This finishes the proof. 
 \end{proof}

\section{Endo-classes of self-dual strata}

Now we compare self-dual semisimple strata which are allowed to correspond to different hermitian spaces where we also allow 
orthogonal and symplectic spaces over~$F$. We further want to allow the strata to have different parameter~$r$ and different~$E_i$-periods, so we repeat the notion of group level and degree. The~\emph{degree} of a semi-pure stratum~$\Delta$ is the~$F$-dimension of~$E$ and the \emph{group level} of a semisimple stratum~$\Delta$ is~$\lfloor\frac{r}{e(\Lambda|E)}\rfloor$ where~$e(\Lambda|E)$ is the lowest common multiple of~$(e(\Lambda^i,E_i))_{i\in I}$ except for the null case where we set the group level to be infinity. 
Two semisimple strata~$\Delta$ and~$\Delta'$ (possibly given on different vector spaces and different skew-fields over~$F$) of the same group level and the same degree are called~\emph{endo-equivalent} if there is a bijection~$\zeta: I\ra I'$ such that~$\Res_F(\Delta_i)\oplus\Res_F(\Delta_{\zeta(i)})$ is equivalent to a simple stratum for all indexes~$i\in I$, 
see~\cite[6.6,6.7]{skodlerack:17-1}. We call~$\zeta_{\Delta,\Delta'}$ the matching from~$\Delta$ to~$\Delta'$.  
This is an equivalence relation by~\cite[,6.7]{skodlerack:17-1}, and note that the used direct sum is a slight generalization, see~\cite[after 6.3]{skodlerack:17-1}. 
The equivalence classes are called~\emph{endo-classes} and the endo-class of a semisimple stratum~$\Delta$ is denoted by~$\mathfrak{E}(\Delta)$. 

\begin{proposition}[\cite{skodlerack:17-1}~6.7]\label{propEndoStrataDegrees}
 Two endo-equivalent strata~$\Delta$ and~$\Delta'$ with matching~$\zeta$ satisfy 
 \[e(\Lambda^i|E_i)=e(\Lambda'^{\zeta(i)}|E'_{\zeta(i)}),\ f(\Lambda^i|E_i)=f(\Lambda'^{\zeta(i)}|E'_{\zeta(i)}),\] for all~$i\in I$. 
\end{proposition}


For endo-equivalence to make sense it is important that two intertwining semisimple strata of same degree and group level
are endo-equivalent. This is true because the  matching results remain true.
\begin{theorem}\label{thmMatchingsDegree&GroupLevel}
 Theorem~\ref{thmMatchingStrataGL}, and Corollary~\ref{corIsometriesForMatchings} remain true if we replace 
 the conditions:~$n=n'$,~$r=r'$ and~$e(\Lambda|F)=e(\Lambda'|F)$ by the assumption that~$\Delta$ and~$\Delta'$ have the same degree and 
 the same group level. 
\end{theorem}

We come to the proof after some preparation. 

\begin{remark}\label{remRaisingRepeating}
There are two steps to reduce from same group level and same degree to equal parameters. 
\begin{enumerate}
 \item[Step 1:] The method of~\emph{repeating}, sometimes also called doubling. Given a stratum~$\Delta$ and a positive integer~$k$
 we can scale the~$F$-period of~$\Lambda$ to~$ke(\Lambda|F)$ in the following way. One constructs a new lattice sequence~$k\Lambda$ by repeating~$k$-times the lattices which occur in the image of~$\Lambda$: 
 \[k\Lambda_j:=\Lambda_{\lfloor \frac{j}{k}\rfloor},\ j\in\ZZ.\]
 The strata~$[k\Lambda,\max(nk,kr+j),kr+j,\beta]$,~$j=0,\ldots,k-1$ define the same coset as~$\Delta$ and 
 they are semisimple if and only if~$\Delta$ is too. The case~$j=0$ is used for the definition of the direct sum and~$j=k-1$
 is useful if one does not want to leave the class of strata which satisfy~$n=r+1$. One important consequence: Given two 
 strata~$\Delta$ and~$\Delta'$ we can assume without loss of generality that both have the same~$F$-period, and if both are 
 non-null, semisimple and intertwine then we obtain~$n=n'$  by~\cite[6.9]{skodlerackStevens:16}. 
 This process does not change the group level and the degree. 
 
 \item[Step 2:] Raising or lowering the parameter~$r$. 
 Suppose~$\Delta$ and~$\Delta'$ are two intertwining semisimple strata of the same degree and the same group level, which share 
 the~$F$-period and the parameter~$n$. Suppose further~$r\geq r'$. Then the stratum~$\Delta'((r-r')+)$, i.e. the stratum obtained from~$\Delta'$
 in raising~$r'$ to~$r$, is still semisimple, by the proof of~\cite[5.48]{skodlerack:17-1}, and we have~$e(\Lambda|E)=e(\Lambda'|E')$
 by~\cite[4.2]{skodlerack:17-1}. Thus~$\Delta$ and~$\Delta'((r-r')+)$ share the group level and the degree. 
 Instead of raising~$r'$ we can lower~$r$ in~$\Delta$ and the next lemma states that~$\Delta((r-r')-)$ and~$\Delta'$ still intertwine. 
\end{enumerate}
An endo-class is invariant under repeating, raising and lowering. 
\end{remark}

\begin{lemma}\label{lemIntertwining&Lowering}
 Suppose~$\Delta$ and~$\Delta'$ are two intertwining semisimple strata sharing the~$F$-period, the parameter~$n$, the group level 
 and the degree. Suppose $r\geq r'$. Then~$\Delta((r-r')-)$ and~$\Delta'$ intertwine. 
 If further both strata are self-dual and intertwine by some element of~$G$, then~$\Delta((r-r')-)$ and~$\Delta'$ intertwine
 by some element of~$G$. 
\end{lemma}

\begin{proof}
 If the common group level is infinite then the identity is a possible intertwiner. So let us suppose that both strata
 have finite group level, i.e. are non-null. Then all strata
 \[\Delta((r-r')-),\Delta,\Delta'((r-r')+),\Delta'\]
 have the same group level, because~$e(\Lambda|E)=e(\Lambda'|E')$, see~\ref{remRaisingRepeating} Step 2 above, 
 and two successive strata are endo-equivalent, the middle two by Theorem~\ref{thmMatchingStrataGL}. Thus, by transitivity, the strata~$\Delta((r-r')-)$ and~$\Delta'$ are endo-equivalent, and therefore intertwine by~\cite[4.32]{skodlerack:17-1}. 
 
 In the self-dual case we apply diagonalization~\ref{propDiagonalizationStrataG} on the endo-equivalent strata~$\Delta((r-r')-)$ and~$\Delta'$ to reduce to the case where~$\beta$ and~$\beta'$ have the same characteristic polynomial. So, if~$\Delta$ 
 and~$\Delta'((r-r')+)$ intertwine by some element of~$G$ then~$\beta$ and~$\beta'$ are conjugate by some element of~$G$, by 
 Corollary~\ref{corSkolemNoetherSemisimple} and therefore~$\Delta((r-r')-)$ and~$\Delta'$ intertwine by some element of~$G$. 
\end{proof}

\begin{proof}[Proof of Theorem~\ref{thmMatchingsDegree&GroupLevel}]
 This Theorem follows now from Remark~\ref{remRaisingRepeating} and Lemma~\ref{lemIntertwining&Lowering}. We leave this to the reader as
 an exercise. 
\end{proof}

In the following, if we say ``possibly  different signed hermitian spaces'' it includes possibly different vector spaces and 
possibly different skew-fields. 

\begin{proposition}
 Let~$\Delta$ and~$\Delta'$ be two endo-equivalent self-dual strata on possibly different signed hermitian spaces~$h$ and~$h'$
 of the first kind and of the same type. 
 Then the matching is equivariant with respect to the adjoint involutions. 
\end{proposition}

\begin{proof}
 We need to consider the restrictions of~$\Delta$ and~$\Delta'$ to~$F$. So for this proof we just assume that 
 both strata are strata over~$F$. By repeating and raising we can assume that both strata share the~$F$-period,
 and the parameters~$n$ and~$r$. By diagonalization, see Proposition~\ref{propDiagonalizationStrataG} for~$i\in I_0$ and 
 Proposition~\ref{propDiagonalizationStrataGL} for~$i\in I_+$, we can assume that~$\beta$ and~$\beta'$ have the same minimal 
 polynomial. For~$i\in I$ we have that~$\beta_i$ and~$\beta'_{\zeta(i)}$ have the same minimal polynomial, and 
 therefore~$-\sigma_h(\beta_i)$ and~$-\sigma_{h'}(\beta_{\zeta(i)})$, i.e.~$\beta_{\sigma(i)}$ and~$\beta'_{\sigma'(\zeta(i))}$,
 have the same minimal polynomial. Thus~$\zeta\circ\sigma=\sigma'\circ\zeta$. 
\end{proof}

We denote the class of self-dual semisimple strata which are endo-equivalent to a given self-dual semisimple stratum~$\Delta$ 
by~$\mathfrak{E}_-(\Delta)$.

\section{Self-dual semisimple  characters}
\newcommand{\tiDelta}{\ti{\Delta}}
\newcommand{\tisigma}{\ti{\sigma}}
\newcommand{\titheta}{\ti{\theta}}
\newcommand{\tibeta}{\ti{\beta}}

In this section we study the first building block for explicit constructions of cuspidal irreducible representations of~$G$. 
These building blocks are the self-dual semisimple characters. We are going to see that they correspond to the semisimple characters
of~$\it{G}$ which are fixed by the adjoint involution. After defining them we turn directly to transfers, followed by the matching 
theory and results on intertwining, diagonalization theory  and conjugacy.

\subsection{First definitions}

Here we  define self-dual semisimple characters. There is a big introduction about semisimple characters in section 5 in~\cite{skodlerack:17-1}, and we use the notations and definitions from there. We fix an additive character~$\psi_F$ from~$F$ to~$\CC^\times$.  One attaches to semisimple stratum~$\Delta$  a set~$\C(\Delta)$ of complex valued characters on a compact open subgroup~$H(\Delta)$ of~$\ti{G}$. Given a character~$\chi$ on a subgroup of~$\it{G}$ we write~$\sigma.\chi$ for~$\chi\circ\sigma.$ Recall that~$p$ is the odd residue characteristic of~$F$.

\begin{definition}\label{defSelfdualSemisChar}
 Let~$\Delta$ be a self-dual semisimple stratum. Then~$H(\Delta)$ and~$\C(\Delta)$ are invariant under the action of~$\sigma$.
 We define~$\C_-(\Delta)$ to be the set of all restrictions of the elements of~$\C(\Delta)^\sigma$, i.e. the~$\sigma$-fixed elements of~$\C(\Delta)$, to~$H_-(\Delta):=H(\Delta)\cap G$. The restriction map is a bijection by Glaubermann-correspondence using that $H(\Delta)$ is a pro-$p$-group with odd~$p$. We call an element~$\theta\in\C(\Delta)^\sigma$ the~\emph{lift} of~$\theta|_{H_-(\Delta)}$ to~$H(\Delta)$. The elements of~$\C_-(\Delta)$ are called~\emph{self-dual semisimple characters}. For a stratum~$\Delta'$ equivalent to~$\Delta$ we define~$\C(\Delta'):=\C(\Delta)$ and~$\C_-(\Delta'):=\C_-(\Delta)$. 
\end{definition}

\begin{proposition}
 Let~$\Delta$ be a self-dual semisimple stratum and~$\theta\in\C(\Delta)$. Then~$\theta|_{H_-(\Delta)}$ is an element of~$\C_-(\Delta)$. 
\end{proposition}

\begin{proof}
 If~$\theta_1,\theta_2$ and~$\theta_3$ are elements of~$\C(\Delta)$ then~$\theta_1\theta_2^{-1}\theta_3$ is an element of~$\C(\Delta)$ by the definition of~$\C(\Delta)$, see~\cite[5.6]{skodlerack:17-1}. 
 Thus the number of elements of~$\C(\Delta)$ which have the same restriction as~$\theta$ on~$H_-(\Delta)$ is independent~$\theta$.
 The number of elements of~$\C(\Delta)$ is a power of~$p$ and thus there is a non-negative integer~$s$ such that for every element 
 of~$\theta\in\C(\Delta)$ there are exactly~$p^s$ elements in~$\C(\Delta)$ with the same restriction as~$\theta$. 
 Thus the action of~$\sigma$ on the set of those characters must have a fixed point. This finishes the proof. 
\end{proof}

The group~$\ti{G}$ acts by conjugation on the set of all complex characters~$\chi$ on subgroups of~$\ti{G}$ which we denote 
by~$(g,\chi)\mapsto g.\chi$. An element~$g\in\ti{G}$ is said to intertwine a character~$\chi:\ K\ra \CC$ with a 
second character~$\chi':\ K'\ra \CC$ if the restrictions of~$\chi'$ and~$g.\chi$ coincide on~$gKg^{-1}\cap K'$. 
We denote the set of all elements of~$\ti{G}$ which intertwine~$\chi$ with~$\chi'$ by~$I(\chi,\chi')$. We adapt the 
notation~$I_H(\chi,\chi')$ for the intersection of~$I(\chi,\chi')$ with a subgroup~$H$ of~$\ti{G}$. 
We say that~$\chi$ and~$\chi'$ intertwine by some element of~$H$ if~$I_H(\chi,\chi')$ is non-empty. 


\subsection{Transfers}
We now recall the notion of transfer. For more detail consult section~\cite[section 6.2]{skodlerack:17-1}. Transfers are defined between strata over different vectors spaces over different skew-fields central and of finite degree over~$F$. 
Let~$\Delta$ and~$\Delta'$ be two endo-equivalent semisimple strata. Let us at first recall the transfer over the same skew-field. 
The canonically defined restriction maps
\[res_{\Delta\oplus \Delta',\Delta}:\ \C(\Delta\oplus\Delta')\ra\C(\Delta)\]
define a  bijection
\[\tau_{\Delta,\Delta'}:=res_{\Delta\oplus \Delta',\Delta'}\circ res_{\Delta\oplus \Delta',\Delta}^{-1}:\C(\Delta)\ra\C(\Delta')\]
called the~\emph{transfer map} from~$\Delta$ to~$\Delta'$. 

In the case of different skew-fields the transfer map is defined as follows:
\[\tau_{\Delta,\Delta'}(\theta):=\tau_{\Delta\otimes L,\Delta\otimes L}(\theta_L)|_{H(\Delta')},\]
where~$\theta_L$ is any extension of~$\theta$ to~$H(\Delta\otimes L)$. Here we use that~$L$ splits the skew-fields. 
We call~$\tau_{\Delta,\Delta'}(\theta)$ the transfer of~$\theta$ from~$\Delta$ to~$\Delta'$. 
Let us recall that if~$\Delta$ and~$\Delta'$ are endo-equivalent and~$V=V'$ then~$\tau_{\Delta,\Delta'}(\theta)=\theta'$ if
 and only if~$I(\Delta,\Delta')\subseteq I(\theta,\theta')$, see~\cite[6.10, 5.9]{skodlerack:17-1}. 

\begin{proposition}\label{propTransferAndInvolution}
 Transfer commutes with the adjoint involutions, i.e. given two signed hermitian spaces~$(h,V)$ and~$(h',V')$ of the first kind  and the same type with adjoint involutions~$\sigma$ and~$\sigma'$, respectively, and given two endo-equivalent semisimple strata~$\Delta$ and~$\Delta'$
 in~$V$ and~$V'$, respectively, then~$\sigma'\circ\tau_{\Delta,\Delta'}=\tau_{\#\Delta,\#'\Delta'}\circ\sigma$.  
\end{proposition}

\begin{proof}
 Without loss of generality we can assume that we work over the same skew-field, because the transfer is defined by reducing to that case, i.e. from different skew-fields to~$L$. Now the result follows, because the restriction
 maps commute with the adjoint involutions, i.e. $res_{\#\Delta\oplus\#'\Delta',\#'\Delta'}\circ(\sigma\oplus\sigma')=\sigma'\circ res_{\Delta\oplus\Delta',\Delta'}$.
\end{proof}

The transfer map induces for self-dual semisimple strata~$\Delta$ and~$\Delta'$ via the Glauberman correspondence a 
bijection from~$\C_-(\Delta)$ to~$\C_-(\Delta')$ which maps~$\theta_-$ 
to~$\tau_{\Delta,\Delta'}(\theta)|_{H_-(\Delta')}$. We denote this map still by~$\tau_{\Delta,\Delta'}$. 

Let us recall:
\begin{definition}[\cite{skodlerack:17-1}~6.3]\label{defEndoEquivCharacters}
 Two semisimple characters~$\theta\in\C(\Delta)$ and~$\theta'\in\C(\Delta')$ on possibly different vector spaces over possibly different skew-fields over strata with the same degree and the same group level are called~\emph{endo-equivalent} if there are transfers which intertwine, i.e. if there are strata~$\tiDelta\in\mathfrak{E}(\Delta)$ and~$\tiDelta'\in\C(\Delta')$ such 
 that~$\tau_{\Delta,\tiDelta}(\theta)$ and~$\tau_{\Delta',\tiDelta'}(\theta')$ intertwine. 
\end{definition}

\subsection{Matching, diagonalization and intertwining}
Here we repeat the result on matching and slightly generalize the diagonalization theorem from~\cite{skodlerack:17-1} and we 
further prove a diagonalization theorem for self-dual semisimple characters. 
We fix in this section two semisimple strata~$\Delta$ and~$\Delta'$ of the same group level and the same degree and 
two semisimple characters~$\theta\in\C(\Delta)$ and~$\theta'\in\C(\Delta')$. 
The restriction~$\theta$ to~$\Aut_D(V^i)\cap H(\Delta)$ is denoted by~$\theta_i$, and it is a simple character for~$\Delta_i$.

\subsubsection{For~$GL$}

We recall the analogue of Theorem~\ref{thmMatchingStrataGL} for characters. 
\begin{theorem}[\cite{skodlerack:17-1}~5.48,~6.18,\cite{broussousSecherreStevens:12}~1.11]\label{thmMatchingCharGL}
 Suppose~$I(\theta,\theta')\neq \emptyset$. Then there is a unique bijection~$\zeta:\ I\ra I'$ such that
 the set~$I(\theta,\theta')\cap \prod_{i\in I}\Hom_D(V^i,V^\zeta(i))$ is not empty. The latter non-emptiness condition is
 equivalent in saying that for all~$i\in I$ the characters~$\theta_i$ and~$\theta'_{\zeta(i)}$ are endo-equivalent and 
 the~$D$-dimensions of~$V^i$ and~$V^{\zeta(i)}$ agree. 
\end{theorem}

We call~$\zeta$ the~\emph{matching} from~$\theta\in\C(\Delta)$ to~$\theta'\in\C(\Delta')$ and we 
write~$\zeta_{\theta,\Delta,\theta',\Delta'}$, but if there is no reason for confusion then we just skip~$\Delta$ and~$\Delta'$. 
Further we get a map~$\bar{\zeta}_{\theta,\theta'}$ between the residue algebras defined in the similar way as it was done 
for strata, see~\cite[5.52]{skodlerack:17-1}.~$(\zeta,\bar{\zeta})$ is again called the matching pair of~$\theta\in\C(\Delta)$
and~$\theta'\in\C(\Delta')$.

Note that in the case of intertwining strata it is possible that~$\zeta_{\theta,\theta'}$ and~$\zeta_{\Delta,\Delta'}$ do not 
coincide. For an example see~\cite[9.7]{kurinczukSkodlerackStevens:16}. But, if~$\theta'$ is a transfer of~$\theta$ from~$\Delta$
to~$\Delta'$, then both matching pairs coincide.

Diagonalization theorems are important to reduce a statement about semisimple characters to the case of transfers. Here is 
the~$\GL$-version. 
\begin{theorem}\label{thmDiagonalizationCharGL}
  Suppose~$\theta$ and~$\theta'$ are endo-equivalent and suppose that the strata share the~$F$-period and the parameters~$r$ and~$n$.
  Then there is a semisimple stratum~$\Delta''$ on~$V'$ with the same associate splitting as~$\Delta'$,~$\Lambda''=\Lambda'$,~$n''=n$ 
  and~$r''=r$ such that~$\Delta\oplus\Delta''$ is semisimple and~$\theta\otimes\theta'\in\C(\Delta\oplus\Delta'')$. 
\end{theorem}

The assertions of theorem imply that~$\beta$ and~$\beta''$ have the same minimal polynomial, because~$\Delta\oplus\Delta''$ is semisimple, 
and~$\theta$ and~$\theta'$ are endo-equivalent. 
 
\begin{proof}
 See the proof of~\cite[5.42]{skodlerack:17-1} and use Theorem~\cite[6.18]{skodlerack:17-1} instead of~\cite[5.35]{skodlerack:17-1}, 
 to get~$\Delta''=[\Lambda',n,r,\beta'']$ such that~$\beta''$ has the same minimal polynomial as~$\beta$ and the 
 same associate splitting as~$\beta'$ and such that $\theta'$ is the transfer of~$\theta$ from~$\Delta$ to~$\Delta''$. 
 Then~$\theta\otimes\theta'\in\C(\Delta\oplus\Delta'')$. 
\end{proof}

We are now able to state the result on intertwining for semisimple characters. 
\begin{theorem}\label{thmIntertwiningCharGL}
 Suppose are~$\Delta$ and~$\Delta'$ are semisimple strata which share the~$F$-period and the parameters~$n$ and~$r$ and 
 suppose that~$\theta\in\C(\Delta)$ and~$\theta'\in\C(\Delta')$ intertwine with matching~$\zeta$.  Then the following holds. 
 \begin{enumerate}
  \item $I(\theta, \theta')= S(\Delta')(I(\theta,\theta')\cap \prod_{i\in I}\Hom_D(V^i,V^{\zeta(i)})) S(\Delta)$.
  \item Suppose~$I(\Delta,\Delta')\neq\emptyset$ and~$\theta'=\tau_{\Delta,\Delta'}(\theta)$, then we have
   \[I(\theta,\theta')= S(\Delta')I(\Delta,\Delta') S(\Delta).\]
  If further there is an element~$\ti{g}$ of~$\ti{G}$ such that~$\ti{g}\beta\ti{g}^{-1}=\beta'$, then we have
  \begin{equation}\label{eqIntCharbetaGL}
   I(\theta,\theta')= S(\Delta')\ti{g}C_A(\beta)^\times S(\Delta).
  \end{equation}
 \end{enumerate}
\end{theorem}

\begin{proof}
The proof of Theorem~\ref{thmIntertwiningCharGL} is the same as for Theorem~\ref{thmIntertwiningStrataGL} 
using~\cite[5.15]{skodlerack:17-1} and Proposition~\ref{thmDiagonalizationCharGL} instead of~\cite[4.36]{skodlerack:17-1} 
and Proposition~\ref{propDiagonalizationStrataGL}. We use that the set~$S(\Delta)$ only depends on~$\theta,\ \Lambda$ and~$r$ because it 
coincides with the intersection of~$P_{\min(-(k_0+r),\lfloor \frac{k_0+1}{2}\rfloor)}(\Lambda)$ with~$I(\theta)$. 
\end{proof}

The next proposition is needed to establish the diagonalization theorem for self-dual characters later. 
\begin{proposition}\label{propSelfdualProplimit}
 Suppose that~$\Delta$ is a stratum such that~$\C(\Delta)=\C(\#\Delta)$,~$\Lambda$ is self-dual and~$-\sigma_h(\beta)$ has the same minimal polynomial and  associate splitting as~$\beta$. 
 Suppose further that~$\tau_{\Delta,\#\Delta}$ is the identity map. Then, there is an element~$\gamma$ 
 of~$A$ with  the same minimal polynomial and  associate splitting as~$\beta$ such that~$[\Lambda,n,r,\gamma]$ is self-dual 
 with the same set of semisimple characters as~$\Delta$. 
\end{proposition}

\begin{proof}
 We write~$\Delta'$ for~$\#\Delta$. 
 We assume~$I=I'$ without loss of generality. The character~$\theta\in\C(\Delta)$ is intertwined by~$1$.  
 Thus~$\bar{\zeta}_{\theta,\Delta,\theta,\Delta'}=\id_{\kappa_E}$. Further~$\theta$ is its transfer from~$\Delta$ 
 to~$\Delta'$ and thus
 \[\id_I=\zeta_{\theta,\Delta,\theta,\Delta'}=\zeta_{\Delta,\Delta'}\]
 and therefore~$\beta_i$ and~$\beta'_i$ have the same minimal polynomial, and
 \[\id_{\kappa_E}=\bar{\zeta}_{\theta,\Delta,\theta,\Delta'}=\bar{\zeta}_{\theta,\tau_{\Delta,\Delta'}(\theta)}=
 \bar{\zeta}_{\Delta,\Delta'}.\]
 Thus by Theorem~\ref{thmIntertwiningStrataGL} there is an element ~$u$ of~$P_1(\Lambda)$ which conjugates~$\beta$ to~$\beta'$ and by~\eqref{eqIntCharbetaGL} (for~$\theta,\Delta\Delta,\ti{g}=1$) we can assume that~$u$ is an element of~$S(\Delta)$ and therefore the~$S(\Delta)$-conjugacy 
 class of~$\beta$ is invariant under the map~$-\sigma_h$. Using the latter action we use the pro-$p$-group property 
 to show  by induction that there is a sequence~$(s_k)_{k\in\NN}$ in~$(\prod_iA^{ii})\cap S(\Delta)$ such that 
 \[-\sigma_h(s_ks_{k-1}\cdots s_1\beta (s_k s_{k-1}\cdots s_1)^{-1})\equiv s_ks_{k-1}\cdots s_1\beta (s_k s_{k-1}\cdots s_1)^{-1}=:\beta^{(k)}~\text{mod }\mf{a}_{1+k}\]
 and 
 \[\beta^{(k)}\equiv \beta^{(k-1)}~\text{mod }\mf{a}_{k}.\]
 
 But then~$(\beta^{(k)})_k$ converges to some element~$\gamma$ and by compactness there is a convergent subsequence~$s_{k_j}$,
 say with limit~$s_0$. We obtain by continuity:~$\gamma=s_0\beta s_0^{-1}$. This finishes the proof. 
\end{proof}

\subsubsection{For~$G$}

In this section we assume that both strata~$\Delta$ and~$\Delta'$ are self-dual and~$\theta\in\C(\Delta)^\sigma$ 
and~$\theta'\in\C(\Delta')^{\sigma'}$. 
We also have a~$G$-version of the diagonalization theorem. Note that there is no~$G$ intertwining required in the 
assumption: 

\begin{theorem}\label{thmDiagonalizationCharG}
Under the assumptions of Theorem~\ref{thmDiagonalizationCharGL} we further assume
that~$\Lambda\oplus\Lambda'$ is self-dual. Then there exists a stratum~$\Delta''$ with~$\Lambda''=\Lambda'$,~$n''=n$ and~$r''=r$ 
such that $\Delta\oplus\Delta''$ is a self-dual semisimple stratum and~$\theta\otimes\theta'\in\C(\Delta\oplus\Delta'')$. 
\end{theorem}

\begin{proof}[Proof of Theorem~\ref{thmDiagonalizationCharG}]
 We can choose a stratum~$\Delta''$ which satisfies the assertions of Theorem~\ref{thmDiagonalizationCharGL}.
 The character~$\theta\otimes\theta'$ is fixed under~$\tisigma$
 by~\cite[5.41]{skodlerack:17-1}, in particular we have~$\C(\Delta\oplus\Delta'')=\C(\#(\Delta\oplus\Delta''))$ and therefore~$\C(\Delta'')=\C(\#\Delta'')$. The elements~$\beta'',\beta$ and~$-\sigma_{h'}(\beta'')$ have the same minimal polynomial. 
 The character~$\theta'$ is a restriction of~$\theta\otimes\theta'$ to~$\Delta''$ and to~$\#\Delta''$ , and thus the transfer map from~$\Delta''$ to~$\#\Delta''$ is the identity map. Now Lemma~\ref{propSelfdualProplimit} finishes the proof. 
\end{proof}

\begin{remark}\label{remDiagCharG}
 Starting with two self-dual lattice sequences~$\Lambda$ and~$\Lambda'$ of the same~$F$-period we can always 
 obtain~$(2\Lambda_j)^\#_j=2\Lambda_{1-j}$ for all~$j\in\ZZ$ after a translation of the domain of~$\Lambda$ and similar for~$\Lambda'$ 
 So one can establish that~$2\Lambda\oplus 2\Lambda'$ is self-dual. 
\end{remark}


The same track to Theorem~\ref{thmIntertwiningStrataG} leads to: 

\begin{theorem}[see~\cite{kurinczukSkodlerackStevens:16}~9.2,9.3 for the case over~$F$]\label{thmIntertwiningCharG}
  Suppose~$\Delta$ and~$\Delta'$ are self-dual and share the~$F$-period and the parameters~$n$ and~$r$ and 
  suppose~$\theta\in\C(\Delta)^\sigma$ and~$\theta'\in\C(\Delta')^{\sigma}$ intertwine by an element of~$G$ with 
  matching~$\zeta$. Then the following holds.
\begin{enumerate}
\item\label{thmIntertwiningCharG.i} $I_G(\theta,\theta')=(G\cap S(\Delta'))(I_G(\theta,\theta')\cap \prod_{i\in I}\Hom_D(V^i,V^{\zeta(i)}))(G\cap S(\Delta))$.
\item Suppose~$\Delta$ and~$\Delta$ intertwine by an element of~$\ti{G}$ and~$\theta'=\tau_{\Delta,\Delta'}(\theta)$, then we have
\begin{equation}\label{eqGIntCharDelta} 
 I_G(\theta,\theta')=(G\cap S(\Delta')) I_G(\Delta,\Delta')(G\cap S(\Delta)).
 \end{equation}
If further there is an element~$\ti{g}$ of~$\ti{G}$ such that~$\ti{g}\beta\ti{g}^{-1}=\beta'$ then
\begin{equation}\label{eqGIntCharbeta} 
I_G(\theta,\theta')=(G\cap S(\Delta'))(G\cap \ti{g}C_A(\beta)^\times)(G\cap S(\Delta)).
\end{equation}
\end{enumerate}
\end{theorem}

\begin{proof}
 The proof is a complete analogue of the proof of Theorem~\ref{thmIntertwiningStrataG} except for~\eqref{eqGIntCharDelta},
 which follows from diagonalization, see Proposition~\ref{propDiagonalizationStrataG},~\eqref{eqGIntCharbeta} 
 and Theorem~\ref{thmIntertwiningStrataG}. 
\end{proof}

\begin{corollary}\label{corMatchingCharG}
Suppose~$\Delta$ and~$\Delta'$ are self-dual and~$\theta\in\C(\Delta)^\sigma$ and~$\theta'\in\C(\Delta')^\sigma$. 
Suppose~$I_G(\theta,\theta')\neq\emptyset$ with matching~$\zeta$.
Then~$I_G(\theta,\theta')\cap\prod_{i\in I}\Hom_D(V^i,V^{\zeta(i)})$ is non-empty  
and the~$\epsilon$-hermitian spaces~$h_i$ and~$h_{\zeta(i)}$ are isometric for every~$i\in I_0$.  
\end{corollary}

\begin{proof}
 By repeating and lowering, Remark~\ref{remDiagCharG} and Theorem~\ref{thmDiagonalizationCharG} we can assume that~$\beta$ and~$\beta'$ have the same minimal polynomial and that~$\theta'=\tau_{\Delta,\Delta'}(\theta)$. 
 After raising, i.e. we consider the third parameter~$\max(r,r')$, formula~\eqref{eqGIntCharbeta}
 implies that~$\beta$ is conjugate to~$\beta'$ over some element of~$G$. Thus~$I_G(\theta,\theta')$ contains an element of~$\prod_{i\in I}\Hom_D(V^i,V^{\zeta(i)})$. this finishes the proof.
\end{proof}

\begin{corollary}\label{corTransIntertwiningCharG}
 Intertwining is an equivalence relation on the set of self-dual semisimple characters of~$G$ of the same degree and the same group level. 
\end{corollary}

\begin{proof}
 Suppose~$\theta_-\in\C_-(\Delta)$,~$\theta'_-\C_-(\Delta')$ and~$\theta''\in\C(\Delta'')$ are three self-dual semisimple characters
 such that~$I_G(\theta_-,\theta'_-)\neq\emptyset$ and~$I_G(\theta'_-,\theta''_-)\neq\emptyset$. Then by Glauberman 
 correspondence~$I_G(\theta,\theta')$ and~$I_G(\theta',\theta'')$ are not empty. By Remark~\ref{remDiagCharG} and Theorem~\ref{thmDiagonalizationCharG} (using repeating and lowering) we can assume without loss of generality that~$\beta$,~$\beta'$ and~$\beta''$ have the same minimal polynomials and that the three characters are transfers of each other. From Theorem~\ref{thmIntertwiningCharG} follows now (after repeating and raising) that~$\beta$,~$\beta'$ and~$\beta''$ are conjugate
 under some elements of~$G$. Thus, as transfers, the characters~$\theta$ and~$\theta''$ intertwine by some element of~$G$. 
\end{proof}

\subsection{Intertwining implies conjugacy for self-dual semisimple characters}
%
%

Another important application of Theorem~\ref{thmDiagonalizationCharG} is called intertwining and conjugacy theorem for self-dual semisimple characters.
\begin{theorem}\label{thmIntConCharG}
Suppose~$\Delta$ and~$\Delta'$ are self-dual semisimple strata sharing the~$F$-period and the parameter~$r$ and~$\theta\in\C(\Delta)^\sigma$ and~$\theta'\in\C(\Delta')^\sigma$ intertwine by an element of~$G$. Let~$(\zeta,\bar{\zeta})$ be their matching pair. 
 Let~$g$ be an element of~$G\cap\prod_i\Hom_D(V^i,V^{\zeta(i)})$ which satisfies~$g\Lambda=\Lambda'$ and such that 
 the conjugation with~$g$ verifies~$\bar{\zeta}|_{\kappa_{E_D}}$.
 Then there is an element~$u$ of~$G\cap \prod_i\Hom_D(V^i,V^{\zeta(i)})$ such that~$u\Lambda=\Lambda'$ and~$u\theta=\theta'$. We can choose~$u$ to further satisfy~$u\beta u^{-1}=\beta'$ if~$\beta$ and~$\beta'$ have the same minimal
 polynomial over~$F$ and~$\tau_{\Delta,\Delta'}(\theta)=\theta'$.  
\end{theorem}

\begin{proof}
 By diagonalization, see Theorem~\ref{thmDiagonalizationCharG}, we can assume that~$\beta$ and~$\beta'$ have the same minimal polynomial and~$\theta'=\tau_{\Delta,\Delta'}(\theta)$. The set~$I_G(\Delta,\Delta')$ is not empty by Theorem~\ref{thmIntertwiningCharG} and from the transfer property follows that~$(\theta,\theta')$ and~$(\Delta,\Delta')$ have the 
 same matching pair. Now the theorem follows from Theorem~\ref{thmIntConStrataG}.
\end{proof}

At the end of this section we show that~$\sigma$-fixed  semisimple character are always lifts of semisimple characters:
\begin{proposition}\label{propSigmaFixedCharacters}
Suppose~$\Delta$ is a semisimple stratum such~$\C(\Delta)$ is invariant under the action of~$\sigma$. 
Then there is a self-dual semisimple stratum~$\Delta'=[\Lambda,n,r,\beta']$ such that~$\C(\Delta)=\C(\Delta')$.
\end{proposition}

\begin{proof}
 The proof is done by induction on~$n-r$. If~$n=r$, then~$\beta=0$ by the definition of semisimple strata, and~$\Delta$ is therefore self-dual. Suppose~$n>r$. By induction hypothesis there is a self-dual semisimple stratum~$\tiDelta:=[\Lambda,n,r+1,\gamma]$
 such that~$\C(\tiDelta)=\C(\Delta(1+))$. By the translation principle~\cite[5.43]{skodlerack:17-1} we can assume without loss of generality that~$\Delta(1+)$ is equivalent to~$\tiDelta$. We take semisimple characters~$\theta\in\C(\Delta)^\sigma$ 
 and~$\titheta\in\C(\tiDelta(1-))^\sigma$. Then there is an element~$a\in\mf{a}_{-r-1}$ such that~$\theta=\psi_a\titheta$. 
 Thus
 \[\titheta\psi_a=\theta=\sigma.\theta=\sigma.\titheta\psi_{-\sigma_h(a)}=\titheta\psi_{-\sigma_h(a)}.\]
 Thus~$\psi_a$ is equal to~$\psi_{a_-}$ for~$a_-:=\frac{a-\sigma_h(a)}{2}$. 
 We are going to prove that~$\Delta'':=[\Lambda,n,r,\gamma+a_-]$ is equivalent to a semisimple stratum~$\Delta'$, because then we can take
 ~$\Delta'$ to be self-dual by~\ref{corSemisimpleEquiv} and we get:
 \[\C(\Delta')=\C(\tiDelta)\psi_{a_-}=\C(\Delta).\]
 We prove that the derived stratum~$\partial_\gamma(\Delta'')$ is equivalent to a semisimple multi-stratum, 
 see~\cite[4.14, 4.15]{skodlerack:17-1} for the definition and the usage of the derived multi-stratum.
 The stratum~$\partial_\gamma(\Delta)$ is a semisimple multi-stratum, see Theorem~\cite[4.15]{skodlerack:17-1}, because~$\Delta$ is semisimple. Let~$s_\gamma$ be the chosen corestriction for~$\gamma$ (for forming the derived strata).  
 Then~$s_\gamma(a_-+\gamma-\beta)$ is modulo~$\mf{a}_{-r}$ congruent to an element of~$F[\gamma]$ because it is intertwined 
 by every element of~$C_A^\times(\gamma)$ because 
 \[\C(\tiDelta(1-))=\C(\tiDelta(1-))\psi_{a_-+\gamma-\beta}.\]
 Thus~$\partial_\gamma(\Delta'')$ is equivalent to a semisimple multi-stratum. Thus~$\Delta''$ is equivalent to a semisimple stratum 
 by~\cite[4.15]{skodlerack:17-1}.
\end{proof}



\section{Self-dual semisimple pss-characters}

In this section we answer the question, when two transfers of endo-equivalent self-dual semisimple characters intertwine over 
an element of the corresponding classical group.We will introduce endo-parameters, as in~\cite{kurinczukSkodlerackStevens:16}. 

\subsection{Self-dual pss-characters}
To introduce self-dual pss-characters we will take a slightly different path than in~\cite{kurinczukSkodlerackStevens:16} because they should not depend on~$\epsilon$. 
We consider all objects of the category~$\mathcal{H}$ of orthogonal and symplectic hermitian forms over~$(D,\rho)$ and~$(F,\id)$.
We start with what should be the domain of a self-dual semisimple character: 
Let~$(\Delta,h)$ be a pair consisting of~$h\in\mc{H}$ and a self-dual semisimple stratum with respect to~$h$,
we also say~\emph{$h$-self-dual}. We define the~\emph{self-dual endo-class of~$\Delta$} as the following class: 
 \[\mathfrak{E}_-(\Delta):=\{(\Delta',h')\in\mathfrak{E}(\Delta)\times\mathcal{H}|\ \Delta'\text{ is~$h'$-self-dual}\}.\]

%
And we define now self-dual pss-character. For the definition of a pss-character (potentially semisimple character) we refer to~\cite[6.3]{skodlerack:17-1}, which is motivated 
by~\cite{broussousSecherreStevens:12}. 
\begin{definition}\label{defselfdualpss}
 Let~$\mathfrak{E}_-$ be a self-dual endo-class. A self-dual pss-character is a map 
 \[\Theta_-:\ \mathfrak{E}_-\ra\bigcup_{(\Delta,h)\in\mathcal{E}_-}\C_-(\Delta),\]
 such that~$\Theta_-(\Delta,h)\in\C_-(\Delta)$ for all~$(\Delta,h)$ and such that the values are related by transfer, i.e. 
 \[\tau_{\Delta,\Delta'}(\Theta_-(\Delta,h))=\Theta_-(\Delta,h'),\]
 for all~$(\Delta,h),(\Delta,h')\in\mathfrak{E}_-$.
\end{definition}

We attach to a self-dual pss-character~$\Theta_-$ with domain~$\mathfrak{E}_-$ the 
pss-character~$\Theta$ whose domain~$\mathfrak{E}$ contains a stratum~$\Delta$, such that~$(\Delta,h)\in\mathfrak{E}_-$ for 
some~$h\in\mathcal{H}$ and such that~$\Theta(\Delta)$ is the lift of~$\Theta_-(\Delta,h)$. We call~$\Theta$ den lift of~$\Theta_-$. 
Two pss-characters~$\Theta$ and~$\Theta'$ (self-dual pss-characters~$\Theta_-$ and~$\Theta'_-$) of the same degree and the same group 
level are called~\emph{endo-equivalent} if 
there are strata~$\Delta\in\mathfrak{E}$ and~$\Delta'\in\mathfrak{E}'$ (~$(\Delta,h)\in\mathfrak{E}_-$ 
and~$(\Delta',h)\in\mathfrak{E}'_-$) such that~$\Theta(\Delta)$ and~$\Theta'(\Delta')$ intertwine (~$\Theta_-(\Delta,h)$ 
and~$\Theta'_-(\Delta',h)$ intertwine by an element of~$\U(h)$). 

\begin{proposition}[\cite{kurinczukSkodlerackStevens:16}~6.18]\label{propEndoEqpss}
 Two self-dual pss-character of the same degree and the same group level are endo-equivalent if and only 
 if their lifts are endo-equivalent. 
\end{proposition}

The equivalence classes are called~\emph{endo-classes} of pss-characters (self-dual pss-characters). 
A pss-character is called ps-character (potentially simple) character if its values are simple characters.
Similar we define self-dual ps-characters. A self-dual pss-character is called~\emph{elementary} if~$I_\Theta$ is a~$\sigma$-orbit, i.e. either it is simple or~$I_\Theta=I_{\Theta,\pm}$ consisting only of one element. 

\subsection{Idempotents and Witt groups}

We recall the equivalence of categories of hermitian forms. Let~$(E,\sigma_E)$ be a field extension of~$F$ together with 
an~$F$-linear non-trivial involution on~$E$

\begin{proposition}\label{propEquivHermCat}
 Let~$\mc{H}_{\sigma_E\otimes\rho,\epsilon}$ be the category of~$\epsilon$-hermitian~$E\otimes_FD$-forms with respect 
 to~$\sigma_E\otimes\rho$, and let~$\mc{H}_{\sigma_E,\epsilon}$ the category of~$\epsilon$-hermitian~$E$-forms with respect 
 to~$\sigma_E$. Then~$\mc{H}_{\sigma_E\otimes\rho,\epsilon}$ is equivalent to~$\mc{H}_{\sigma_E,\epsilon}$. 
\end{proposition}

\newcommand{\Matr}{\text{M}}
\begin{proof}
 We define a functor~$\mc{F}:\mc{H}_{\sigma_E\otimes\rho,\epsilon}\ra \mc{H}_{\sigma_E,\epsilon}$ in the following way.
 We consider an~$E$-algebra isomorphism~$\Phi:\ E\otimes_FD\cong\Matr_2(E)$. The anti-involution~$\sigma_E\otimes\rho$ is pushed forward to 
 a unitary anti-involution which can be interpreted as the adjoint anti-involution of a unitary form bilinear form. Such a form 
 has a diagonal~$\sigma_E$-symmetric Gram-matrix~$u=\diag(u_1,u_2)$ because the characteristic of~$F$ is odd. So we can choose~$\Phi$ such that 
 the push forward of~$\sigma_E\otimes\rho$ is~$u\sigma_E(\ )^Tu^{-1}$. We identify~$(E\otimes_FD,\sigma_E\otimes_F\rho)$ 
 with~$(\Matr_2(E),u\sigma_E(\ )^Tu^{-1})$. We consider
 \[1^1:=\left(\begin{matrix}1 & 0\\0 & 0\end{matrix}\right)\ 1^2:=\left(\begin{matrix}0 & 0\\0 & 1\end{matrix}\right).\]
 We define
 \newcommand{\tih}{\ti{h}}
 \newcommand{\antidiag}{\text{antidiag}}
 \[\mc{F}(V,\tih):=(V1^1,\tr_E\circ\tih|_{V1^1\times V1^1}).\]
 Exercise: Show that~$\mc{F}(V,\tih)$ is non-degenerate. 
 
 The inverse functor~$\mc{G}$ is given by
 \[\mc{G}(W,h_E):=(W\oplus W,\tih),\ \tih(w_1+w_2,w'_1+w'_2)
 =\left(\begin{matrix}h_E(w_1,w'_1)& h_E(w_1,w'_2)\\u_2u_1^{-1}h_E(w_2 ,w'_1)&u_2u_1^{-1}h_E(w_2,w'_2)\end{matrix}\right).\]
  The~$E\otimes_FD$-action on~$W\oplus W$ is given by
 \[(w_1,w_2)E_{ij}:=(\delta_{j1}w_i,\delta_{j2} w_i),\]
 for~$1\leq i,j\leq 2$, and by the~$E$ action on~$W$. 
\end{proof}

This Proposition shows that the objects~$\mc{H}_{\sigma_E\otimes\rho,\epsilon}$ inherit the decomposition properties from 
the objects of~$\mc{H}_{\sigma_E,\epsilon}$. In particular we have a Witt group of~$E\otimes_FD$ with respect to~$\sigma_E\otimes\rho$ 
and~$\epsilon$, we denote this Witt group by~$W_{\epsilon}(\sigma_E\otimes\rho)$. 

\begin{proposition}\label{propIndOfuDef}
 The Witt-group~$W_{\epsilon}(\sigma_E\otimes\rho)$ is independent of the choice of the isomorphism 
 from~$E\otimes_FD$ to~$\Matr_2(E)$.  
\end{proposition}
\newcommand{\tih}{\ti{h}}
\begin{proof}
 Consider two~$E$-algebra isomorphisms~$\Phi_i:\ (E\otimes_FD,\sigma_E\otimes\rho)\stackrel{\sim}{\longrightarrow} (\Matr_2(E),u_i\sigma_E(\ )^Tu_i^{-1})$, such that~$u_i$ is diagonal and symmetric with respect to~$\sigma_E$. 
 Let the above functor be denoted by~$\mc{F}_i$. Since~$\mc{F}_2\circ\mc{F}^{-1}_1$ is an equivalence, in particular it sends 
 isometric spaces to isometric spaces and respects orthogonal sums, we only need to prove that it sends hyperbolic spaces of dimension~$2$ to hyperbolic spaces of dimension~$2$. We take~$(V,\tih)\in\mc{H}_{\sigma_E\otimes\rho,\epsilon}$ such 
 that~$\mc{F}_1(\tih)$ is a hyperbolic space of dimension~$2$. Then~$\mc{F}_1(\tih)$ has a non-zero isotropic vector and thus~$\tih$ has a non-zero isotropic vector~$v\in V$. This vector generates a simple~$E\otimes_FD$-module~$<v>$ on which~$\tih$ vanishes. Let~$1^1_2$ and~$1^2_2$ be the idempotents for~$u_2$. There is a simple~$E\otimes_FD$-module which does not vanish under~$1^1_2$, see for example~$E\times E$, so~$<v>$ does not vanish under~$1^1_2$, because all simple~$E\otimes_FD$-modules are isomorphic to each other. Thus~$\mc{F}_2(\tih)$ 
 has a non-zero isotropic vector, which finishes the proof. 
\end{proof}

\newcommand{\Idemp}{\text{Idemp}}

The functor~$\mc{F}$ in the Proof of~\ref{propEquivHermCat} is completely characterized by the idempotent~$e$ of~$E\otimes_FD$
which is mapped to~$1^1$ under~$\Phi$, more precisely~$\mc{F}(\tih)=\tr_E\circ\tih|_{Ve}$. We denote the functor associated to~$e$ 
by~$\mc{F}_e$. Let~$\Idemp(\sigma_E\otimes\rho)$ be the set of rank~$1$ idempotents of~$E\otimes_FD$ which are fixed 
by~$\sigma_E\otimes\rho$. Two elements of~$\Idemp(\sigma_E\otimes\rho)$ are conjugate by a similitude~$g$ 
of~$\sigma_E\otimes\rho$, i.e.~$g$ satisfies~$g(\sigma_E\otimes\rho)(g)=s\in E^\times$.

\begin{proposition}\label{propFfFe-1}
 Let~$e,f\in\Idemp(\sigma_E\otimes\rho)$ and let~$g$ be an element of~$(E\otimes_FD)^\times$ such that~$geg^{-1}=f$ 
 and~$g(\sigma_E\otimes_F\id_D)(g)=s\in E^\times$. 
 Then~$\mc{F}_f\circ\mc{F}^{-1}_e$ is equivalent to the functor
 \[h_E\mapsto s^{-1}h_E.\]
\end{proposition}

\begin{proof}
We take forms~$\tih\in\mc{H}_{\sigma_E\otimes\rho,\epsilon}$ and~$h_E\in\mc{H}_{\sigma_E,\epsilon}$ such that~$\mc{F}_e(\tih)=h_E$. 
For~$v_1,v_2\in Vf$ we have:
\begin{eqnarray*}
 \tr_E(\tih(v_1,v_2))&=&\tr_E(f\tih(v_1,v_2)f)\\
 &=&\tr_E(geg^{-1}\tih(v_1,v_2)geg^{-1})\\
 &=&\tr_E(ge\tih(v_1gs^{-1},v_2g)eg^{-1})\\
 &=&s^{-1}h_E(v_1g,v_2g).\\ 
\end{eqnarray*}
Thus~$\mc{F}_f(\tih)$ is isometric to~$s^{-1}h_E$ via~$g:\ Vf\ra Ve$.  
Thus the maps~$g:\ \mc{F}_f(\tih)\ra\mc{F}_e(\tih)$ form a natural transformation from~$\mc{F}_f$ to~$s^{-1}\mc{F}_e$ 
which is an equivalence. Thus~$\mc{F}_f\circ\mc{F}^{-1}_e$ is equivalent to~$s^{-1}\mc{F}_e\circ\mc{F}^{-1}_e$ and the latter is
\[h_E\mapsto s^{-1}h_E\]
by the definition of~$\mc{F}_e^{-1}$ in the proof of Proposition~\ref{propEquivHermCat}. 
\end{proof} 

We are now able to describe the elements of~$W_{\epsilon}(\sigma_E\otimes\rho)$ independent of the choice of an idempotent.
We call a field extension~$E=F[\beta]|F$~\emph{self-dual}  if there is an automorphism~$\sigma_E$ of~$E|F$ which maps~$\beta$ to~$-\beta$. 
\newcommand{\wtower}{\textit{wtower}}
\begin{definition}\label{defWitttower}
 Suppose~$E=F[\beta]$ is a self-dual field extension different from~$F$. 
 A~\emph{Witt tower} for~$\sigma_E,\rho,\epsilon$ is a function \[\wtower:\ \Idemp(\sigma_E\otimes\rho)\ra W_{\epsilon}(\sigma_E)\]
 which satisfies~$\wtower(f)=s^{-1}\wtower(e)$ for all~$e,f\in\Idemp(\sigma_E\otimes\rho)$ using the 
 similitude of Proposition~\ref{propFfFe-1}. If there is no confusion with~$\rho$ and~$\epsilon$ then we also say Witt tower for~$\sigma_E$ or  for~$\beta$. 
\end{definition}

We have a canonical bijection from~$W_{\epsilon}(\sigma_E\otimes\rho)$ to the set of Witt towers for~$\sigma_E,\rho,\epsilon$:
An element~$\tih_\equiv$ is mapped to~$\wtower_{\tih}$ defined via~$\wtower_{\tih}(e)=\mc{F}_e(\tih)_\equiv\in W_{\epsilon}(\sigma_E)$. The map is well-defined by
Proposition~\ref{propIndOfuDef}. 
We identify~$W_{\epsilon}(\sigma_E\otimes\rho)$ with the set of Witt towers for~$\sigma_E,\rho,\epsilon$.
Given a field extension~$(E,\sigma_E)|F$ with non-trivial~$\sigma_E$ and an~$F$-linear~$\sigma_E$-$\id_F$-equivariant non-zero map~$\lambda:E\ra F$ there is a natural map
\[\Tr_{\lambda}:\ W_{\epsilon}(\sigma_E\otimes\rho)\ra W_{\epsilon}(\rho),\ \]
defined via~$Tr_{\lambda}(\tih_\equiv):=((\lambda\otimes_F\id_D)\circ\tih)_\equiv.$

\subsection{Matching Witt towers}

Given an~$\epsilon$-hermitian form~$h$ on~$V$ with respect to~$(D,\rho)$ and an element~$\beta$ in the Lie algebra of~$G$ such 
that~$F[\beta]$ is a field, there are unique forms~$h_\beta\in\mc{H}_{\sigma_E,\epsilon}$ 
and~$\tih_\beta\in\mc{H}_{\sigma_E\otimes\rho,\epsilon}$ such that~$\lambda_\beta\circ h_\beta=\trd_{D|F}\circ h$ 
and~$(\lambda_\beta\otimes_F\id_D)\circ \tih_\beta=h$, where~$\lambda_\beta:\ E\ra F$ is the~$F$-linear extension of~$\id_F$ with 
kernel~$\sum_{i=1}^{[E:F]-1}\beta^iF$.
We have
\[h_\beta=\tr_E\circ\tih_\beta,\]
because:
\begin{lemma}
 $\lambda_\beta\circ\tr_E=\trd_{D|F}\circ (\lambda_\beta\otimes_F\id_D)$.
\end{lemma}

\begin{proof}
 It is a direct consequence of~$\trd_{D|F}(x)=\tr_E(x)$ for all~$x\in D$. 
\end{proof}

We call the function~$\wtower_{\tih_\beta}$ the Witt tower of~$(\beta,h)$ and we also denote it 
by~$\wtower_{\beta,h}$.

We want to be able to compare Witt towers for different~$\beta$ and~$\beta'$.

Let us recall that~$\beta$ is called minimal over~$F$, if~$\nu_{F[\beta]}(\beta)$ is 
co-prime to~$e(F[\beta]|F)$ and~$\beta^{e(F[\beta]|F)}+\mf{p}_{F[\beta]}$ generates~$\kappa_{F[\beta]}$ over~$\kappa_F$. 
Define~$e_p(\beta)$ as the fraction~$\frac{e_{wild}(E|F)}{\gcd(e_{wild}(E|F),\nu_E(\beta))}$ where~$e_{wild}(E|F)$ is the wild ramification index.
Among those field extensions~$E'|F$ in~$E=F[\beta]$ which are generated by an element congruent to~$\beta^{e_p(\beta)}$ 
there is a smallest one. We write~$F[\beta_{min,tr}]$ where~$\beta_{min,tr}$ is a minimal element over~$F$ and
congruent to~$\beta^{e_p(\beta)}$. In saying~$\beta_{min,tr}$ we mean a choice. This choice is not canonical and can be changed to our purpose 
if necessary, for example if~$\beta$ is skew-symmetric, we choose~$\beta_{min,tr}$ skew-symmetric (w.r.t~$\sigma_E$).

So let us assume that~$E=F[\beta]$ and~$E'=F[\beta']$ are self-dual field extensions with non-zero~$\beta$ and~$\beta'$. 
We are going to define a map from~$W_{\epsilon}(\sigma_E\otimes\rho)$ to~$W_{\epsilon}(\sigma_{E'}\otimes\rho)$.
This is a generalization of~\cite[5.2]{kurinczukSkodlerackStevens:16}.
Let us recall them: 
\[w_{\beta,\beta'}:\ W_{-1}(\sigma_E)\ra W_{-1}(\sigma_{E'}),\ w_{\beta^2,\beta'^2}:\ W_{1}(\sigma_E)\ra W_{1}(\sigma_{E'})\]
are bijections which respect the anisotropic dimension and satisfy
\[w_{\beta,\beta'}(\langle \beta\rangle_\equiv)=\langle\beta'\rangle_\equiv,\ w_{\beta^2,\beta'^2}(\langle \beta^2\rangle_\equiv)=\langle\beta'^2\rangle_\equiv.\]
Let us write~$w_{\beta,\beta',1}$ for~$w_{\beta^2,\beta'^2}$ and~$w_{\beta,\beta',-1}$ for~$w_{\beta,\beta'}$.
%
We need two further assumptions: 
\begin{enumerate}
 \item[(A)]\label{itemCondA} Suppose there is an isomorphism from~$F[\beta_{min,tr}]|F$ to~$F[\beta'_{min,tr}]|F$ which sends~$\beta_{min,tr}$ to an element congruent to~$\beta'_{min,tr}$, we write~``$\beta_{min.tr}\mapsto\beta'_{min,tr}$'' for this unique isomorphism.
 \item[(B)]\label{itemCondB} Suppose that~$F[\beta_{min,tr}]_0\subseteq N_{E|E_0}(E)$ if and only if~$F[\beta'_{min,tr}]_0\subseteq N_{E'|E'_0}(E')$. (This is a numerical condition as the next lemma states.) 
\end{enumerate}

We define the bijection
$w_{\beta,\beta',\epsilon}:\ W_{\epsilon}(\sigma_E\otimes\rho)\ra W_{\epsilon}(\sigma_{E'}\otimes\rho)$ as follows: Let~$
e\in\Idemp(\sigma_E|_{F[\beta_{min,tr}]}\otimes_F\rho)$ and~$e'\in\Idemp(\sigma_{E'}|_{F[\beta'_{min,tr}]}\otimes_F\rho)$ 
be idempotents such that~$e\mapsto e'$ under \[F[\beta_{min,tr}]\otimes_FD\stackrel{\text{``}\beta_{min,tr}\mapsto\beta'_{min,tr}\text{''}}{\longrightarrow}F[\beta'_{min,tr}]\otimes_FD.\]
We define~$w_{\beta,\beta',\epsilon}(\wtower)$ to be the Witt tower of~$\sigma_{E'}\otimes\rho$ which satisfies \[w_{\beta,\beta',\epsilon}(\wtower)(e'):=w_{\beta,\beta',\epsilon}(\wtower(e)).\]
This map is well-defined, i.e. does not depend on the choice of the idempotent, by Proposition~\ref{propFfFe-1} and~(B).


\begin{lemma}\label{lemNormNonnorm}
 Let~$=F[\beta]|F$ be a field extension with a non-trivial involution~$\sigma_E$ given by~$\beta\mapsto -\beta$. 
 Then~$F[\beta_{min,tr}]_0$ is a subset of~$\N_{E|E_0}(E)$ if and only if~$E|F[\beta_{min,tr}]_0$ has even 
 ramification index and even inertia degree. If we do not have the above containment, then~$x\in F[\beta_{min,tr}]_0$ is 
 a norm of~$F[\beta_{min,tr}]|F[\beta_{min,tr}]_0$ if and only if it is a norm of~$E|E_0$. 
\end{lemma}

\begin{proof}
 We just write~$E'$ for~$F[\beta_{min,tr}]$. At first we remark that a norm of~$E'|E'_0$ is a norm of~$E|E_0$ because both have degree two and~$\sigma_E$ and~$\sigma_E|_{E'}$ are the Galois generators. Here we use that the residue characteristic is odd. 
 Suppose now that~$E|E'_0$ has even ramification index and even inertia degree.  Then every element of~$o_{E'_0}^\times$ is a square in~$E$ and therefore a norm of~$E|E_0$. If~$E'|E'_0$ is ramified then~$E'_0$ contains a uniformizer which is a norm of~$E'|E'_0$, and we have the desired containment. In the case of an unramified~$E'|E'_0$, since~$e(E|E'_0)$ is even there is a uniformizer of the maximal unramified extension 
 in~$E_0|E'_0$ which is a square in~$E$, and since~$-1$ is a norm of~$E|E_0$ we obtain that the mentioned uniformizer is a norm of~$E|E_0$. 
 All elements of~$o^\times_{E_0}$ are norms of~$E|E_0$ because~$E|E_0$ is also unramified. Thus there is a uniformizer of~$E'_0$ which is a norm of~$E|E_0$. 
 Suppose for the converse that all elements of~$E'_0$ are norms of~$E|E_0$. Then all elements of~$\kappa_{E'_0}$ are squares 
 in~$\kappa_E$, i.e.~$f(E|E'_0)$ is even, and there is a uniformizer of~$E'_0$ which is a norm of~$E|E_0$, in particular the~$\nu_E$-valuation of this uniformizer must  be even and therefore~$e(E|E'_0)$ is even. 
 In the case where~$E'_0$ is not contained in~$\N_{E|E_0}(E)$, say~$x\in E'_0$ is not a norm of~$E|E_0$, the set of all non-norms of~$E'|E'_0$ which is~$x\N_{E'|E'_0}(E'^\times)$ is disjoint to~$\N_{E|E_0}(E)$ which finishes the proof.  
\end{proof}

\begin{definition}\label{defMatchingWitttowers}
 let~$h$ and~$h'$ be two~$\epsilon$ hermitian forms and suppose that~$\beta\in \Lie(U(h))$ and~$\beta'\in \Lie(U(h'))$ generate field extensions~$E$ and~$E'$ different from~$F$, such that~(A) and~(B) hold.
 We say that the Witt towers of~$(\beta,h)$ and~$(\beta',h')$~\emph{match} if~$h$ is isometric to~$h'$ and~$w_{\beta,\beta',\epsilon}((\tih_\beta)_\equiv)=(\tih_{\beta'})_\equiv$. Note that matching Witt towers is an equivalence relation on
 the pairs~$(\beta,h)$ and not on Witt towers. 
\end{definition}

%
%

\begin{remark}\label{rematchingWittTowersfromconjugation}
 If~$g: h\cong h'$ is an isometry then the Witt towers of~$(\beta,h)$ and~$(g\beta g^{-1},h')$ match.
\end{remark}

\begin{proof}
 We consider~$\tih:=\tih_\beta$ and~$\tih':=c_g\circ\tih\circ (g^{-1}\times g^{-1})$. We have to show~$\tih'=\tih'_{\beta'}$.
 We write~$\phi(x)$ for~$gxg^{-1}$,~$x\in E$.
 For~$v_1,v_2\in V$ we have: 
 \begin{eqnarray*}
  \tih'(v_1,v_2\phi(x))&=& g\circ\tih(g^{-1}(v_1),g^{-1}(v_2\phi(x)))\circ g^{-1}\\
  &=& g\circ\tih(g^{-1}(v_1),(g^{-1}\circ \phi(x))(v_2))\circ g^{-1}\\
&=& g\circ\tih(g^{-1}(v_1),(x\circ g^{-1})(v_2))\circ g^{-1}\\
&=& g\circ\tih(g^{-1}(v_1),g^{-1}(v_2)x)\circ g^{-1}\\
&=& g\circ\tih(g^{-1}(v_1),g^{-1}(v_2))\circ x\circ g^{-1}\\
&=&\tih'(v_1,v_2)\phi(x)\\
 \end{eqnarray*}
 Proceeding further this way we see that~$\tih'$ is an~$\epsilon$-hermitian form which satisfies~$(\lambda_{\beta'}\otimes\id_D)\circ\tih'=h'$ which finishes the proof. 
\end{proof}

\begin{theorem}\label{thmMatchingWitttowersIntertwining}
 Let~$\Delta$ and~$\Delta'$ be two non-null self-dual simple strata on~$(V,h)$ and~$\theta_-\in\C_-(\Delta)$ and~$\theta'_-\in\C_-(\Delta')$
 be two endo-equivalent self-dual simple characters. Then~$\theta_-$ and~$\theta'_-$ intertwine by an element of~$G$ if and only if 
 the Witt towers of~$(\beta,h)$ and~$(\beta',h)$ match. 
\end{theorem}

We need for the proof the twist~$h^\gamma$ of a signed hermitian form~$h$ by a skew-symmetric or symmetric element~$\gamma\in\ti{G}$, see subsection~\ref{subsecWitt}. If~$\gamma$ is skew-symmetric and invertible then~$h$ is symplectic 
if and only if~$h^\gamma$ is orthogonal

\begin{proof}
 We choose lifts~$\theta$ and~$\theta'$ for~$\theta_-$ and~$\theta'_-$. 
 The field extensions~$F[\beta_{min,tr}]|F$ and~$F[\beta'_{min,tr}]|F$ are isomorphic by a~$\sigma_h$-equivariant map which 
 sends~$\beta_{min,tr}$ to an element congruent to~$\beta'_{min,tr}$. so we can assume without loss of generality that~$\beta'_{min,tr}$
 is the image of~$\beta_{min,tr}$ under this isomorphism. 
 Suppose at first~$I_G(\theta,\theta')\neq\emptyset$. 
 Then~$\beta_{min,tr}$ and~$\beta'_{min,tr}$ are conjugate by an element of~$G$. 
 Thus we can assume~$\beta_{min,tr}=\beta'_{min,tr}$ without loss of generality. (Note that conjugation with an element of~$G$ does not
 change the equivalence class of~$(\beta,h)$ by Remark~\ref{rematchingWittTowersfromconjugation}.) Now take any idempotent~$e$ of~$F[\beta_{min,tr}]\otimes_FD$ and
 choose~$(\tiDelta,\trd_{D|F}\circ h|_{Ve})\in\mf{E}_-(\Delta,h)$ and~$(\tiDelta',\trd_{D|F}\circ h|_{Ve})\in\mf{E}_-(\Delta',h)$
 such that~$\tibeta=e\beta$ and~$\tibeta'=e\beta'$. The transfers of~$\theta$ and~$\theta'$ to~$\tiDelta$ and~$\tiDelta'$, 
 respectively, intertwine by an element of~$\Aut_F(Ve)$ by Theorem~\cite[6.1]{skodlerack:17-1}. 
 Proposition~\cite[5.3]{kurinczukSkodlerackStevens:16} implies for the orthogonal case that the Witt towers 
 of~$(\tibeta,\trd_{D|F}\circ h|_{Ve})$ and~$(\tibeta',\trd_{D|F}\circ h|_{Ve})$ match and thus the Witt towers of~$(\beta,h)$ 
 and~$(\beta',h)$ match. 
 In the symplectic case we consider the twist~$h^{\beta_{min,tr}}$. The latter is orthogonal and the argument of part one shows 
 that the Witt towers of~$(\beta,h^{\beta_{min,tr}})$ and~$(\beta',h^{\beta_{min,tr}})$ match. 
 Thus~$\wtower_{\beta,h^{\beta_{min,tr}}}(e)$ and~$\wtower_{\beta',h^{\beta_{min,tr}}}(e)$
 are mapped to each other under~$W_{1}(\sigma_{E})\cong W_{1}(\sigma_{E'})$ via~$\langle \beta^2\rangle_\equiv\mapsto
 \langle \beta'^2\rangle_\equiv$.
 The twist with~$\beta_{min,tr}$ induces the map
 \[W_{-1}(\sigma_{E})\ra W_{-1}(\sigma_{E}),\ \langle \beta\rangle_\equiv\mapsto\langle \beta^{e_p(\beta)+1}\rangle_\equiv=
\langle (-1)^{\frac{e_p(\beta)-1}2}\beta^2\rangle_\equiv, \]
 and we have the analogue formula for~$\beta'$. Now~$-1$ is a norm of~$E|E_0$ if and only if it is a norm of~$E'|E'_0$ because
 both~$E|F$ and~$E'|F$ have the same inertia degree. Thus the pull back of the map~$W_{1}(\sigma_{E})\cong W_{1}(\sigma_{E'})$ 
 under the twist is
 the map~$W_{-1}(\sigma_{E})\cong W_{-1}(\sigma_{E'})$,~$\langle \beta\rangle_\equiv\mapsto\langle\beta'\rangle_\equiv$.
 We obtain that the Witt towers of~$(\beta,h)$ and~$(\beta',h)$ match. 
 
 We now consider the converse. So suppose that the Witt towers of~$(\beta,h)$ and~$(\beta',h)$ match. 
 Without loss of generality we can assume that~$\beta$ and~$\beta'$ have the same minimal polynomial 
 by Theorem~\ref{thmDiagonalizationCharG} and that the characters are transfers of each other. 
 Note that this did not leave the equivalence class of~$(\beta,h)$
 by the first direction of this proof. We have~$e\in\Idemp(\sigma_E\otimes\rho)$ and~$e'\in\Idemp(\sigma_{E'}\otimes_F\rho)$ 
 which are mapped under $\text{``}\beta_{min,tr}\mapsto \beta'_{min,tr}\text{''}$ to each other 
 such that~$\mc{F}_e(\tih_\beta)_\equiv$ is mapped to~$\mc{F}_{e'}(\tih_{\beta'})$ 
 under~$W_{\epsilon}(\sigma_E)\cong W_{\epsilon}(\sigma_{E'})$. We consider the pullback~$\phi^*\tih_{\beta'}$ of~$\tih_{\beta'}$ 
 along the field-isomorphism~$\phi: E\ra E'$ with~$\phi(\beta)=\beta'$. Then the matching Witt tower condition says that
 ~$\mc{F}_e(\tih_\beta)$ and~$\mc{F}_e(\phi^*\tih_{\beta'})$ are isometric. Thus~$\tih_{\beta}$ and~$\phi^*\tih_{\beta'}$ are 
 isometric and thus there is an element~$g$ of~$G$ which conjugates~$\beta$ to~$\beta'$. And this intertwines~$\theta$ with~$\theta'$
because the two characters are transfers of each other.  
\end{proof}

The last proof shows a nicer version for an equivalent criteria for~$G$-intertwining.
\begin{corollary}\label{corIntertwMinimalCriteria}
 Let~$\Delta$ and~$\Delta'$ be two self-dual non-null simple strata for~$(V,h)$ 
 and~$\theta\in\C(\Delta)^\sigma$ and~$\theta'\in\C(\Delta')^\sigma$ be two endo-equivalent simple characters. Then the following
 conditions are equivalent: 
 \begin{enumerate}
  \item $I_G(\theta,\theta')\neq \emptyset$ \label{corIntertwMinimalCriteria-i}
  \item $\Delta((n-r-1)+)$  and~$\Delta'((n'-r'-1)+)$ intertwine under some element of~$G$.\label{corIntertwMinimalCriteria-ii}
  \item $(\beta_{min,tr},h)$ and~$(\beta'_{min,tr},h)$  have matching Witt towers. \label{corIntertwMinimalCriteria-iii}
  \item $(\beta,h)$ and~$(\beta',h)$ have matching Witt towers. \label{corIntertwMinimalCriteria-iv}
 \end{enumerate} 
\end{corollary}

\begin{proof}
 The first and the last assertion are equivalent by 
 Theorem~\ref{thmMatchingWitttowersIntertwining} and~\cite[7.1]{kurinczukSkodlerackStevens:16},~\ref{corIntertwMinimalCriteria-i} implies \ref{corIntertwMinimalCriteria-ii}, and further the assertion~\ref{corIntertwMinimalCriteria-ii} implies~\ref{corIntertwMinimalCriteria-iii} by the implication
 \ref{corIntertwMinimalCriteria-i}$\Rightarrow $\ref{corIntertwMinimalCriteria-iv} for 
 strata (equivalently: characters which are transfers). 
 So let us assume~\ref{corIntertwMinimalCriteria-iii}. We can assume without loss of generality that~$\beta'_{min,tr}$ is the image 
 of~$\beta_{min,tr}$ under~$F[\beta_{min,tr}]|F\cong F[\beta'_{min,tr}]|F$. Then matching Witt towers implies that~$\beta_{min,tr}$ 
 and~$\beta'_{min,tr}$ are conjugate by an element of~$G$, so we can assume without loss of generality that~$\beta_{min,tr}$ and~$\beta'_{min,tr}$ 
 coincide. Now we proceed as in the proof of Theorem~\ref{thmMatchingWitttowersIntertwining} to conclude from the 
 endo-equivalence of~$\theta$ with~$\theta'$ that the Witt towers of~$(\beta,h)$ and~$(\beta',h)$ match. 
\end{proof}

We now generalize the formalism of~\cite{kurinczukSkodlerackStevens:16} which leads to endo-parameters
for quaterionic inner forms of classical groups. 

For that we need to include the case~$\beta=0$ in the definition of matching Witt towers. We call~$h_\equiv\in W_{\epsilon}(\id_F)$ the Witt tower 
of~$(0,h)$. We say that the Witt towers of~$(0,h)$ and~$(0,h')$ match if~$h$ is isometric to~$h'$. Let us recall we can identify the index sets of strata of the domain of a pss-character to one set~$I_\Theta$, and endo-equivalent pss-characters~$\Theta$ and~$\Theta'$ determine a matching~$\zeta:\ I_\Theta\ra I_{\Theta'}$, see Theorem~\cite[6.18]{skodlerack:17-1}. 
Given lifts~$\Theta$ and~$\Theta'$ of self-dual pss-characters the index set decomposes into~$I_\Theta=I_{\Theta,0}\cup I_{\Theta,\pm}$, and we sometimes choose 
a disjoint union~$I_{\Theta,\pm}=I_{\Theta,+}\cup I_{\Theta,-}$ such that~$\sigma(I_{\Theta,+})=I_{\Theta,-}$.  
Let us recall that given a bijection~$\zeta: I_\Theta\ra I_{\Theta'}$, a~$\zeta$-comparison pair is a pair~$(\Delta,\Delta')\in \mf{E}\times\mf{E}'$ such that both strata are defined over the same skew-field~$D$ and~$\dim_DV^i=\dim_DV^{\zeta(i)}$ for all~$i\in I_\Theta$, see~\cite[6.17]{skodlerack:17-1}.

\begin{theorem}\label{thmEndoequivalentselfdualpss}
 Let~$\Theta_-$ on~$\mf{E}_-$ and~$\Theta'_-$ on~$\mf{E}'_-$ be two endo-equivalent self-dual pss-characters with lifts~$\Theta$ and~$\Theta'$ and matching~$\zeta=\zeta_{\Theta,\Theta'}$. Then for given pairs~$(\Delta,h)\in\mf{E}_-$ and~$(\Delta',h)\in \mf{E}_-$ are equivalent:
 \begin{enumerate}
  \item $\Theta_-(\Delta,h)$ and~$\Theta_-(\Delta',h)$ intertwine by an element of~$U(h)$.\label{thmEndoequivalentselfdualpss-i}
  \item $(\Delta,\Delta')$ is a~$\zeta$-comparison pair, and the Witt towers of~$(\beta_i,h_i)$ 
  and~$(\beta'_{\zeta(i)},h_{\zeta(i)})$ match, for all~$i\in I_{\Theta,0}$. \label{thmEndoequivalentselfdualpss-ii}
 \end{enumerate}
\end{theorem}

\begin{proof}
 The direction~\ref{thmEndoequivalentselfdualpss-i}$\Rightarrow$\ref{thmEndoequivalentselfdualpss-ii} follows from the
 Corollaries~\ref{corMatchingCharG} and~\ref{corIntertwMinimalCriteria}. Backwards follows block-wise from Theorem~\cite[6.18]{skodlerack:17-1} and Corollary~\ref{corIntertwMinimalCriteria}.
\end{proof}

Finally we can introduce endo-parameters for quaterionic inner forms of classical groups. 

\subsection{Endo-parameters}\label{sectionEndoParameter}

At first we generalize the notion of Witt type from~\cite[Section 13.2]{kurinczukSkodlerackStevens:16}. 
We fix~$\rho$ and~$\epsilon$.
We consider pairs~$(\beta,t)$ where~$F[\beta]$ is a self-dual field extension and~$t\in W_\epsilon(\sigma_E\otimes\rho)$.
We call~$(\beta,t)$ equivalent to~$(\beta',t')$ if at least one of the following points hold:
\begin{itemize}
 \item $t$ and~$t'$ are hyperbolic.
 \item $\beta$ and~$\beta'$ are non-zero,~$\Tr_{\lambda_\beta}(t)=\Tr_{\lambda_{\beta'}}(t')$ and the Witt towers match, i.e.~$w_{\beta,\beta',\epsilon}(t)=t'$. 
 \item $\beta=\beta'=0,\ t=t'$ and~$t$ is not hyperbolic. 
\end{itemize}
The equivalence  classes of these pairs are called~$(\rho,\epsilon)$-\emph{Witt types} and the factor set is denoted by~$\mc{W}_{\rho,\epsilon}$. The Witt type of the pairs~$(\beta,t)$ with~$t$ hyperbolic is denoted by~$0$.

The second data needed for endo-parameters are certain endo-classes: A semisimple character~$\theta\in\C(\Delta)$ is called full if~$r=0$. A pss-character is called full if its domain contains a stratum with~$r=0$, these are the pss-characters of group level zero or infinity. Similar for self-dual pss-characters. 
We denote by~$\mc{E}$ the set of all full endo-classes of ps-characters and by~$\mc{E}_-$ the set of all 
elementary full endo-classes. 
We denote for a full semisimple character~$\theta\in\C(\Delta)$ and a full endo-class~$c\in\mc{E}$ by~$\theta_c\in\Delta_c$
the restriction of theta to the summand corresponding to~$c$, if all simple endo-classes in~$c$ occur in~$\theta$. 
Similar for self-dual semisimple characters. 
The~\emph{degree} of~$c_-\in\mc{E}_-$ we define to be the degree of a simple block restriction~$c_1$ of~$c_-$ where
the degree of~$c\in\mc{E}$ is the degree of any stratum in the domain an element of~$c$. We write~$\deg(c)$ and~$\deg(c_-)$. 

We say that a Witt type~$[(\beta,t)]$ is a Witt type~\emph{for}~$c_-\in\mc{E}_-$ if either~$c_-$ is simple and there is a simple stratum~$\Delta=[\Lambda,n,0,\beta]$
which occurs as a first coordinate in the domain of some element of~$c_-$ or if~$c_-$ is not simple and~$[(\beta,t)]=0$.

\begin{definition}\label{defEndo}
 A~$(\rho,\epsilon)$-endo-parameter is a map~$f_-=(f_1,f_2):\ \mc{E}_-\ra \bbN\times \mc{W}_{\rho,\epsilon}$
 with finite support, such that~$f_2(c_-)$  is a Witt type for~$c_-$  and~$f_1(c_-)$
 is divisible by~$\frac{\deg(D)}{\gcd(\deg(c_-),\deg(D))}$ for every~$c_-\in\mc{E}_-$ (This divisibility condition is empty for the non-null simple elementary~$c_-$.).
 Recall that a~$\GL$-endo-parameter is just a map~$f:\ \mc{E}\ra \bbN_0$ of finite support. 
 See Definition~\cite[7.1]{skodlerack:17-1}.
\end{definition}

We can attach to a~$(\rho,\epsilon)$-endo-parameter~$f_-$ a~$GL$-endo-parameter~$f$, also called its lift, where we define
for a simple block restriction~$c$ of~$c_-$:
\[f(c)=\left\{\begin{array}{ll}f_1(c_-)&,\text{ if }c_-\text{ is not simple}\\
               2f_1(c_-)+\diman(f_2(c_-))\frac{\deg(D)}{\gcd(\deg(c_-),\deg(D))}&, \text{ else}
              \end{array}\right.\]
where~$\diman(f_2(c_-))$ is the anisotropic~$F[\beta]$-dimension of~$t(e)$ (for any idempotent~$e$) in the Witt type~$f_2(c_-)=[(\beta,t)]$. 
We define
\[\deg(f_-):=\deg(f):=\sum_{c\in\mc{E}}f(c)\deg(c).\]

We attach to any Witt type~$[(\beta,t)]$ the Witt tower~$WT_D([(\beta,t)]):=\Tr_{\lambda_\beta}(t)$. 



\begin{theorem}[see~\cite{kurinczukSkodlerackStevens:16}~13.11, for the~$F$-case]\label{thmEndoparameter}
 The set of intertwining classes of full semisimple characters for~$G=\U(h)$ is in canonical bijection to the 
 set of endo-parameters~$f_-$ which satisfy: 
 \begin{enumerate}
  \item $\deg(f_-)=\deg(\End_D(V))$.
  \item $\sum_{c_-\in\mc{E}_-}WT_D(f_2(c_-))=h_\equiv$.
 \end{enumerate}
 The map is constructed as follows: Given an intertwining class of~$\theta_-\in\C_-(\Delta)$ we define 
 \[f_1(c_-)=\left\{\begin{array}{ll}(\text{Witt index of}~\tih_{\beta_{c_-}})\frac{\deg(D)}{\gcd(\deg(c_-),\deg(D))}& \text{ for simple }c_-\in \mc{E}_-\\
                    \deg(\End_{E_c\otimes_F D}V^c) & \text{ if~$c$ is a block restriction of }c_-\in \mc{E}_-\text{ non-simple.}  
                   \end{array}\right.,\]
 and 
 \[f_2(c_-)=\left\{\begin{array}{ll}\text{Witt type of}~(\beta_{c_-},h|_{V^{c_-}})& \text{ for simple }c_-\in \mc{E}_-\\
                    0 & \text{ for non-simple }c_-\in \mc{E}_-.  
                   \end{array}\right.
 \]
\end{theorem}

One can see the non-simple~$c_-$ as~$GL$-parts of the endo-parameter.  

\begin{proof}
 The map is well defined by Theorem~\ref{thmEndoequivalentselfdualpss}.
 We show at first the injectivity of the map. We consider two full self-dual semisimple characters~$\theta_-\in\C_-(\Delta)$ and~$\theta'_-\in\C_-(\Delta')$ for~$\U(h)$ with 
 the same~$(\rho,\epsilon)$-endo-parameter~$f_-$. Then their lifts~$\theta$ and~$\theta'$ intertwine by Theorem~\cite[7.2]{skodlerack:17-1} because they have the 
 same $GL$-endo-parameter. Now 
 Theorem~\ref{thmEndoequivalentselfdualpss} implies that~$\theta_-$ and~$\theta'_-$ intertwine. Conversely, we have to show that any endo-parameter~$f_-$ of the form 
 given in the theorem is attained by a self-dual semisimple character for~$\U(h)$. 
 We only consider~$c_-$ in the support of~$f_-$. 
 \begin{itemize}
  \item For a simple~$c_-$ we take a full self-dual semisimple character~$\theta_{c_-}\in\C(\Delta_c)$ whose self-dual ps-character is an element of~$c_-$
  such that~$\tih_{\beta_{c_-}}$ has Witt index~$f_1(c_-)\frac{\gcd(\deg(c_-),\deg(D)}{\deg(D)}$ and Witt type~$f_2(c_-)$. 
  \item For a non-simple~$c_-$ we consider a simple block~$c_1$ of~$c_-$. We take a simple character~$\theta_{c_1}\in\C(\Delta_{c_1})$ with 
  endo-parameter supported in~$c_1$ which maps~$c_1$ to~$f_1(c_-)$. We construct a hyperbolic~$\epsilon$-hermitian space~$h_{c_-}$ 
  with Lagrangian~$V^{c_1}$. There is a full semisimple character with block restrictions~$\theta_{c_1}$ and~$\theta_{c_1}^{\sigma}$ by~\cite[7.1]{skodlerack:17-1}. 
  We can take the stratum to be self-dual by Proposition~\ref{propSigmaFixedCharacters}.  
\end{itemize}
Again by~\cite[7.1]{skodlerack:17-1} and~\ref{propSigmaFixedCharacters} there is a self-dual semisimple stratum~$\Delta$ for~$\bigobot_{c_-}h_{c_-}$
such that~$\theta_-:=\otimes_{c_-}\theta_{c_-}\in\C_-(\Delta).$ 
Now~$h$ is isometric to~$\bigobot_{c_-}h_{c_-}$ and we can take an isometry to conjugate~$\theta_-$ to a character for~$h$. By construction and Remark~\ref{rematchingWittTowersfromconjugation} 
the character~$\theta_-$ has endo-parameter~$f_-$.  
\end{proof}


%
%
%
%
%
%

\appendix
\section{Intertwining classes of self-dual embeddings}

In this section we answer the following question. Let us underline that in this section we use that~$D$ is not a field. Say~$\theta_-\in\C_-(\Delta)$ is a full 
self-dual semisimple character with lift~$\theta\in\C_-(\Delta)^{\sigma}$. 
How many~$G$-intertwining classes of~$\sigma$-fixed semisimple characters are contained the~$\ti{G}$-intertwining class of~$\theta$. 
Let~$f_-$ be the endo-parameter of~$\theta_-$ and~$f$ its lift. By Theorem~\ref{thmEndoparameter} we only need to find all endo-parameters~$f'_-$ with 
lift~$f$ and such that~$\sum_{c_-}WT_D(f'_2(c_-))=h_\equiv$. 
These endo-parameters only differ in their values for simple~$c_-\in\mc{E}_-$. 
We start the simple case.
Let us recall: We call an equivariant-$F$-algebra homomorphism~$\phi:\ (E,\sigma_E)\ra (\End_D(V),\sigma_h)$ a \emph{self-dual embedding}. 

\begin{proposition}\label{propSelfdualEmbeddings}
 Let~$E=F[\beta]$ be a self-dual field-extension different from~$F$ and suppose there is a self-dual embedding into~$(\End_D(V),\sigma_h)$.
 Then there are precisely two~$G$-conjugacy classes of self-dual embeddings of~$(E,\sigma_E)$ into~$(\End_D(V),\sigma_h)$.
\end{proposition}

Proposition~\ref{propEquivHermCat} and~\ref{propFfFe-1} are true if one replaces~$\sigma_E$  with~$\id_E$. Note that~$\id_E\otimes_F\rho$ is orthogonal and thus 
the set~$\Idemp(\id_E\otimes_F\rho)$  of~$\id_E\otimes_F\rho$-fixed idempotents of rank~$1$  is non-empty. We can again identify the 
Witt-group~$W_\epsilon(\id_E\otimes_F\rho)$ with the set of Witt towers, defined as in Definition~\ref{defWitttower}, i.e. as maps 
\[\Idemp(\id_E\otimes \rho)\ra W_\epsilon(\id_E),\ \wtower_{\tih}(e):=\mc{F}_e(\tih)_\equiv.\]
For an extension~$(E|\sigma_E)|(E'|\sigma_{E'})$, both of even degree over~$F$ and an~$E'$-linear non-zero map~$\lambda$ 
we get a map\[\Tr_\lambda:\ W_{\epsilon}(\id_E\otimes\rho)\ra W_\epsilon(\id_{E'}\otimes\rho),\]
which in terms of Witt towers is given by
\[\wtower_{\Tr_\lambda(t)}(e')=\Tr_\lambda(\wtower_t(e')),\ e'\in\Idemp(\id_{E'}\otimes_F\rho).\]
where~$\Tr_\lambda:\ W_\epsilon(\sigma_E)\ra W_\epsilon(\sigma_{E'})$ is given by~$\tih_\equiv\mapsto (\lambda\circ\tih)_\equiv$.
We are now able to prove Proposition~\ref{propSelfdualEmbeddings}.

\begin{proof}
There are at most two conjugacy classes because the parity of the anisotropic dimension of~$\tih\in W_{\epsilon}(\sigma_e\otimes \rho)$ such 
that~$(\lambda_\beta\otimes\id_D)\circ\tih$ is isometric to~$h$ is determined by the degree of~$E|F$ and~$\dim_DV$ and there are exactly two Witt towers 
in~$W_\epsilon(\sigma_E\otimes\rho)$ with the same parity for the anisotropic dimension.  
 It is enough to show that the maximal anisotropic Witt tower is mapped to the hyperbolic Witt tower under~$\Tr_{\lambda_\beta}$. 
 $E|F$ contains a~$\sigma_E$-invariant quadratic extension~$(E',\sigma_{E'}=\sigma_E|_{E'})$ of~$F$. The image of the maximal anisotropic Witt tower under~$\Tr_{\lambda} $ 
 does not depend on the choice of~$\lambda:\ (E,\sigma_E) \ra F$, non-zero, equivariant and~$F$-linear, by~\cite[2.2]{kurinczukSkodlerackStevens:16}. 
 So, we could replace~$\lambda_\beta$ by~$\tr_{E'|F}\circ\lambda'$ for some non-zero, equivariant~$E'$-linear~$\lambda':(E,\sigma_E)\ra (E',\sigma_{E'})$. 
 The map~$\Tr_{\lambda'}$ sends the maximal anisotropic Witt tower to the maximal anisotropic one by~\cite[4.4]{skodlerackStevens:16}. 
 So, we have to show that~$\Tr_{E'|F}$ sends the maximal anisotropic Witt tower to the hyperbolic one. The maximal anisotropic Witt tower in~$W_\epsilon(\sigma_{E'}\otimes\rho)$
 can be written in the form~$X=(x\tih)_\equiv+\tih_\equiv$ for some~$\tih_\equiv\in W_\epsilon(\sigma_{E'}\otimes\rho)$ and~$x\in F^\times$. 
 By~\ref{propWittGroupD} the forms~$(\tr_{E'|F}\otimes\id_D)\circ\tih$ and~$x((\tr_{E'|F}\otimes\id_D)\circ\tih)$ are isometric and the Witt group of~$(D,\rho)$ is an
 elementary~$2$-group. So~$\Tr_{E'|F}(X)$ is the hyperbolic Witt tower. 
\end{proof}

Two count the number of~$(\epsilon,\rho)$-endo-parameters~$f'_-$ with the same lift~$f$ we only have to arrange Witt types such that~$\sum_{c_-}WT_D(f'_2(c_-))=h_\equiv$.
The above proposition shows that for the two choices for~$f'_2(c_-)$ we have~$WT_D(f_2(c_-))=WT_D(f'_2(c_-))$. So for every non-null simple~$c_-\in\mc{E}_-$ we have two choices for the Witt type and for the others the Witt type is determined by~$f$. So:

\begin{proposition}\label{propNumberGConjugacyClasses}
 The number of~$\U(h)$-intertwining classes of~$\sigma$-fixed semisimple characters in the~$\ti{G}$-intertwining class of~$\theta$ is equal
 $2^{\# I_0}$ if there is no null block restriction for~$\theta$ and~$2^{\# I_0-1}$ if~$\theta$ has a null block restriction. 
\end{proposition}

\def\Circlearrowleft{\ensuremath{%
  \rotatebox[origin=c]{180}{$\circlearrowleft$}}}

\bibliographystyle{plain}
\bibliography{/home/zdsk/LaTeX/bib/bibliography}

\end{document}